\documentclass{amsart}

\usepackage[latin1]{inputenc}
\usepackage{amsfonts}
\usepackage{amsmath}
\usepackage{amsthm}
\usepackage{amssymb}
\usepackage{latexsym}
\usepackage{enumerate}
\usepackage{accents}

\newtheorem{theorem}{Theorem}
\newtheorem{lemma}[theorem]{Lemma}
\newtheorem{proposition}[theorem]{Proposition}

\newtheorem{remark}[theorem]{Remark}

\newtheorem{question}[theorem]{Question}

\newtheorem{definition}[theorem]{Definition}

\numberwithin{equation}{section}

\begin{document}

\newcommand{\cc}{\mathfrak{c}}
\newcommand{\N}{\mathbb{N}}
\newcommand{\BB}{\mathbb{B}}
\newcommand{\C}{\mathcal{C}}
\newcommand{\J}{\mathcal{J}}
\newcommand{\Q}{\mathbb{Q}}
\newcommand{\R}{\mathbb{R}}
\newcommand{\Z}{\mathbb{Z}}
\newcommand{\T}{\mathbb{T}}
\newcommand{\HH}{\mathbb{H}}
\newcommand{\st}{*}
\newcommand{\PP}{\mathbb{P}}
\newcommand{\rin}{\right\rangle}
\newcommand{\SSS}{\mathbb{S}}
\newcommand{\forces}{\Vdash}
\newcommand{\dom}{\text{dom}}
\newcommand{\osc}{\text{osc}}
\newcommand{\F}{\mathcal{F}}
\newcommand{\A}{\mathcal{A}}
\newcommand{\B}{\mathcal{B}}
\newcommand{\I}{\mathcal{I}}
\newcommand{\X}{{X}}
\newcommand{\Y}{\mathcal{Y}}
\newcommand{\V}{\mathcal{V}}
\newcommand{\XX}{\mathcal{X}}
\newcommand{\CC}{\mathcal{C}}
\newcommand{\norm}[1]{\left\lVert#1\right\rVert}
\newcommand{\coz}{{\mathsf{coz}}}
\newcommand{\supp}{{\mathsf{supp}}}

\thanks{The research of the first named author was supported by the NCN (National Science
Centre, Poland) research grant no.\ 2020/37/B/ST1/02613.}

\title[On subspaces of indecomposable Banach spaces]
{On subspaces of indecomposable Banach spaces}

\author{Piotr Koszmider}
\address{Institute of Mathematics  of the Polish Academy of Sciences,
ul. \'Sniadeckich 8,  00-656 Warszawa, Poland}
\email{\texttt{piotr.math@proton.me}}

\author{Zden\v ek  Silber}
\address{Institute of Mathematics of the Polish Academy of Sciences,
ul. \'Sniadeckich 8,  00-656 Warszawa, Poland}
\email{\texttt{zdesil@seznam.cz}}

\begin{abstract} 
We address the following question: what is the class of  Banach spaces isomorphic to subspaces of indecomposable Banach spaces?
We show that this class includes all Banach spaces of density not bigger than the continuum which
do not admit $\ell_\infty$ as a quotient (equivalently do not admit a subspace isomorphic to $\ell_1(\cc)$).
This includes all Asplund spaces and all weakly Lindel\"of determined Banach spaces of density not
bigger than the continuum.
However, we also show that this class includes some Banach spaces admitting $\ell_\infty$ as a quotient.
This sheds some light on the question asked  in [S. Argyros, R. Haydon, 
\emph{Bourgain-Delbaen $L^\infty$-spaces, the scalar-plus-compact property
and related problems}, Proceedings of the International Congress of Mathematicians (ICM 2018),
Vol. III, 1477--1510. Page 1502] whether all Banach spaces not containing $\ell_\infty$ embed
in some indecomposable Banach spaces.
Our method of constructing indecomposable Banach spaces above
a given Banach space is a considerable modification of the method of constructing Banach spaces
of continuous functions with few$^*$ operators developed before by the first-named author.
\end{abstract}

\maketitle

\section{Introduction}

Resolving a problem of Lindenstrauss from \cite{lind-problem} whether every infinite dimensional Banach space $\XX$ is 
decomposable (i.e., can be written as a direct sum $\XX'\oplus\XX''$, where $\XX', \XX''$ are infinite
dimensional closed subspaces of $\XX$) Gowers and Maurey have constructed in \cite{gm}
Banach spaces where no infinite dimensional closed subspace is decomposable (such spaces will be called here as usual
HI spaces i.e., hereditarily indecomposable spaces). A natural question which arises in
this context which was considered in \cite{argyros-etal,
argyros-haydon, argyros-r} is:

\begin{question}\label{main-question} Which Banach spaces can be 
isomorphic to subspaces of indecomposable Banach spaces?
\end{question}

The authors of the papers \cite{argyros-etal,
argyros-haydon, argyros-r} looked at it mainly in the case of separable Banach spaces. 
In particular, Argyros and Raikoftsalis showed in
\cite{argyros-r} that every separable reflexive Banach space is a subspace of an indecomposable separable Banach space.
However, it is mentioned in \cite{argyros-r} that the authors of this paper together
with R. Haydon constructed 
an indecomposable separable Banach space which contains $\ell_1$. Also in \cite{argyros-r} a problem is posed
if all separable Banach spaces not containing $c_0$ are subspaces of an indecomposable Banach space
(as by the Sobczyk theorem,
separable spaces containing $c_0$ cannot be subspaces of separable indecomposable spaces).

In this paper we address Question \ref{main-question} in the nonseparable case of Banach spaces up to
the density continuum (denoted by $\cc$). For this we use another type 
of infinite dimensional indecomposable Banach spaces constructed
 in \cite{few} by the first-named author, namely
classical Banach spaces of continuous functions on a compact Hausdorff space (we call  them spaces of the
form $C(K)$). Such spaces must be nonseparable and 
cannot be HI as they always contain $c_0$ or even $C([0,1])$, and hence any separable Banach space.
However, they carry much algebraic and lattice structure and can be of arbitrarily large size, at least
under some set-theoretic assumption (\cite{superfew}) unlike HI spaces which all embed into $\ell_\infty$
(\cite{argyros-tolias}, \cite{plichko-yost}). 

We obtain the following partial answer to Question \ref{main-question}:

\begin{theorem}\label{main} If $\XX$ is a Banach space of density not bigger
than the continuum and it does not admit a quotient isomorphic to $\ell_\infty$,
then there is an indecomposable Banach space of density not bigger than the continuum of the form
$C(K)$ for a compact Hausdorff $K$ and an isometric embedding of $\XX$ into $C(K)$.
\end{theorem}
\begin{proof} By a theorem of Talagrand (Proposition \ref{talagrand}) the
hypothesis on $\XX$ implies that the dual ball
$B_{\XX^*}$ considered with the weak$^*$ topology does not
admit a subspace homeomorphic to the \v Cech-Stone compactification $\beta\N$ 
of $\N$.  So, Proposition \ref{cor-topo} implies that $C(B_{\XX^*})$ isometrically
embeds into an  indecomposable Banach space of density at most $\cc$ of the form $C(K)$.
Since $\XX$ isometrically embeds into $C(B_{\XX^*})$, we obtain the theorem.
\end{proof}

As $\beta\N$ does not admit any nontrivial convergent sequence,
 it follows from our result that if  the dual ball $B_{\XX^*}$ of a Banach space $\XX$
of density at most $\cc$
considered with the weak$^*$ topology has some reasonable sequential properties,
then $\XX$ is a subspace of an indecomposable Banach space. This includes
all Asplund Banach spaces or weakly Lindel\"of determined Banach spaces
since their dual balls being Radon-Nikodym compacta or Corson compacta are sequential
(5.4 of \cite{namioka}, 1.6 of \cite{kalenda}).

Our Theorem \ref{main} does not resolve  Question \ref{main-question} entirely, even
for the spaces of density at most $\cc$.
Certainly not all Banach spaces can be subspaces of indecomposable Banach spaces.
For sure, Banach spaces which contain an isomorphic copy of $\ell_\infty$ cannot, because
$\ell_\infty$ is complemented in every superspace and has many decompositions. This  class of spaces
includes all injective Banach spaces by a result of Rosenthal (Corollary 3 of \cite{rosenthal}). 
However, there are many Banach spaces which admit $\ell_\infty$ as their quotient
but not as their subspace. We know that some
of them are subspaces of indecomposable Banach spaces, for example we obtain:

\begin{theorem}\label{ind-quotient} There is an indecomposable Banach space of the form
$C(K)$ for $K$ compact and Hausdorff such that $\beta\N$ is homeomorphic to
a subspace of $K$, and so $\ell_\infty$ is isometric to a quotient of $C(K)$.
\end{theorem}
\begin{proof}
    Let $K$ be the compact space from Theorem 5.1 of \cite{few}. 
    Then $C(K)$ is indecomposable, so we will be done if we show that  $K$ admits  a 
    subspace homeomorphic to $\beta\N$.

    By Proposition \ref{weak-subseq-prop},
    it is enough to show that $C(K)$ has the weak subsequential separation property of Definition \ref{def-ssp}.
    But this is true, as by Theorem 5.1 of \cite{few},
    the space $C(K)$ has a stronger property that for every pairwise disjoint sequence 
    $(f_n)_{n \in \N}$ in $C_1(K)$ there is infinite $M \subseteq \N$
    such that $\sup_{n \in M} f_n$ exists in $C_1(K)$. Here, as elsewhere in the paper,
    $C_1(K)$ stands for the set of all functions in $C(K)$ which have their values in $[0,1]$.
\end{proof}

As seen from the proof  above, Theorem \ref{ind-quotient} is witnessed already by one of the spaces constructed in \cite{few}.
To prove it we generalize a result of \cite{alg-universalis} providing an interesting by itself  sufficient condition
for a compact $K$ to contain a copy of $\beta\N$ (Propositions \ref{weak-subseq-prop} and \ref{subseq-top}).
In fact, assuming Martin's axiom and the negation of the continuum hypothesis
every indecomposable Banach space of the form $C(K)$ admits $\ell_\infty$
as a quotient. This follows from the result of \cite{hlo} and the fact that
not being  Grothendieck for a space of the form $C(K)$ is equivalent to
admitting a complemented copy of $c_0$ (\cite{schachermayer}).
On the other hand, we do not know any example
of  a Banach space not containing an isomorphic copy
of $\ell_\infty$ which cannot be a subspace of an indecomposable Banach space.
So the following question formulated by S. Argyros and R. Haydon remains open:

\begin{question}[\cite{argyros-haydon} p. 1502] Is every Banach space 
not containing $\ell_\infty$ a subspace of an indecomposable Banach space?
\end{question}

We do not know the answer either in the case of Banach spaces of  density at most $\cc$.
A relevant example of a Banach space not containing $\ell_\infty$
but admitting it as many of its  quotients  could be  the space of the
form $C(K)$ where $K$ is the compact $F$-space  form \cite{antonio-pk} which exists under Martin's axiom
and the negation of the continuum hypothesis. We do not know if it can
be a subspace of an indecomposable Banach space or even a subspace
an indecomposable $C(K)$-space.

Let us comment on the methods we use to construct the indecomposable space of  Theorem \ref{main}
and at the same time lets us discuss the structure of the paper. 
For this let us recall the terminology of \cite{superfew}: A Banach space
of the form $C(K)$ is said to have few$^*$ operators if for every linear bounded operator $T$ on it
 the adjoint operator $T^*$
on the dual $C(K)^*$ (identified with the space of Radon measures on $K$) 
is of the form $gI+S$, where $g: K\rightarrow \R$ is
a Borel function and $S$ is a weakly compact operator on $C(K)^*$.
A Banach space
of the form $C(K)$ is said to have few operators if  every linear bounded operator $T$ on it
 is of the form $fI+S$, where $f\in C(K)$ and $S$ is a weakly compact on $C(K)$. In both definitions
 $I$ stands for the identity operator on $C(K)^*$ and on $C(K)$ respectively.
 It was proved in Theorem 2.5 of \cite{few} that if  $K\setminus F$ is connected
 for all finite $F\subseteq K$ and the  space $C(K)$ has few$^*$ operators, then it is indecomposable, and it follows by 
 a simpler argument (cf. Theorem 12.3 of \cite{fewsur}) that 
  if $K$ is connected and $C(K)$ has few operators, then it is indecomposable.
 
Our indecomposable Banach space of the form $C(K)$ from Theorem \ref{main} has few$^*$ operators
and we do not know if such space can be obtained to have few operators. 
All constructions of  spaces with few operators  in the above sense in the literature
take place as constructions of subalgebras  with few$^*$ operators in some ambient 
algebras  like $\ell_\infty$ in \cite{few}, $L_\infty(\{0,1\}^\cc)$ in \cite{plebanek-few}
or $C(G(\{0,1\}^\cc))$, where $G(\{0,1\}^\cc)$ is the Gleason space of $\{0,1\}^\cc$, in \cite{superfew}.
All these spaces are of the form $C(M)$ for $M$ extremally disconnected and c.c.c.
The c.c.c is used very much to pass from few$^*$ to few operators (in \cite{few} only under {\sf CH}).
Such spaces cannot
contain, for example, the space $c_0(\cc)$ so are less interesting for us  in the context of embedding
many spaces of density $\mathfrak c$. So we rather adopt
the approach of \cite{antonio} where the ambient space is $C(\beta\N\setminus\N)$,
however for the same reasons as in \cite{antonio}, in the absence of the c.c.c.  we are able to
obtain only spaces with few$^*$ operators and not with few operators.

To embed a given Banach space $\XX$ into a space of the form $C(K)$ with few$^*$ operators 
instead of $C(\beta\N\setminus \N)$ we use as the ambient space $C(M)$, where
$M$ is the \v Cech-Stone remainder of $B_{\XX^*}\times\R^2$, that is
$\beta(B_{\XX^*}\times\R^2)\setminus B_{\XX^*}\times\R^2$, where $B_{\XX^*}$
is the dual ball of $\XX$ with the weak$^*$ topology. Such compact
space maps continuously onto $B_{\XX^*}$ and so $\XX$ isometrically
embeds into $C(M)$. Like in the case of $\beta\N\setminus\N$, our $M$
is an almost P-space, realcompact and F-space, so some methods of \cite{antonio} can be generalized.
Sections 3 and 4 are devoted to developing these techniques.
The disadvantage of working with such spaces
as opposed to extremally disconnected spaces is that 
$C(M)$ is not a complete lattice, so we do not always have suprema of bounded sets.
One thus needs to use submorphisms (Subsection 4.2)  as in \cite{antonio} which can serve in a similar way to
the use of lattice completeness. In particular, we obtain
a generalized version of Theorem 3.1 of \cite{antonio} in Proposition \ref{few-star}
which provides a sufficient condition for a Banach space of the form $C(K)$
to have few$^*$ operators in the contexts of $K$'s like connected \v Cech-Stone reminders.

To build a subalgebra of $C(M)$ of the form $C(K)$ which is indecomposable and moreover contains a copy
of a given $\XX$ we need  $C(K)$  to satisfy the requirements of Proposition
\ref{few-star}, we need $K\setminus F$ to be connected for every finite $F\subseteq K$ and
moreover we need to include a given ahead $\XX$ in $C(K)$.  This challenge 
has not been yet addressed in the literature. Here we propose new methods.
First, the inductive construction of $C(K)$ becomes possible  due to
 the hypothesis that $B_{\XX^*}$ does not contain $\beta\N$ and
a  strengthening of a method of Haydon \cite{haydon}
needed for the inductive step of our construction.
The main difference is that the  algebra $\A$ of our Lemma \ref{haydon} (which contains,
in the final construction, the algebra $C(B_{\XX^*})$) may have density $\cc$,
while the Boolean algebra of Lemma 1D of \cite{haydon} needs to have cardinality smaller than $\cc$.
The second challenge is to deal with the connectedness requirement in this context.
For this we use a completely new approach, namely the high level of the connectedness
of the \v Cech-Stone remainder $M$ of $B_{\XX^*}\times\R^2$. It
turns out that if a subalgebra of the form $C(K)$ of
$C(M)$ is large enough (we call it abundant, see Definition \ref{def-abundant}), then
$K\setminus F$ is connected for every finite $F\subseteq K$ (Lemma \ref{lemma-abundant-connected}) as required for
the indecomposability of $C(K)$ in the case of few$^*$ operators.
When the ambient spaces in \cite{few}, \cite{plebanek-few} and \cite{superfew} were 
of the form $C(M)$ for $M$ extremally disconnected we needed to build
a ``small'' subalgebra $C(K)$ to achieve the connectedness of $K$. Here
 the ambient space is $C(M)$ for $M$ connected, and so we need to build its  ``big'' subalgebra $C(K)$ 
to achieve the connectedness of $K$.

The results dealing with the connectedness are developed in Section 5. Section 6 is devoted to
the actual construction of the indecomposable space of Theorem \ref{main}. Section 7 includes
the results needed for Theorem \ref{ind-quotient}. In the following Section 2 we discuss 
the terminology and notation which we use in the paper.

At the end of the introduction let us mention some natural questions related to Question \ref{main-question}.
Our Theorem \ref{main} produces possibly distinct indecomposable Banach spaces for
each  Banach space whose dual ball does not admit
a subspace homeomorphic to $\beta\N$. We do not know if it is possible to have  some 
types of universal spaces with these properties:

\begin{question} Let  $\kappa$ be an uncountable cardinal. Let
$\mathcal{SI}_\kappa$ be the class of all Banach spaces 
which are subspaces of some indecomposable Banach spaces of density $\kappa$.
 Is it consistent that there is an indecomposable Banach space of density $\kappa$ which admits
subspaces isomorphic to all members of $\mathcal{SI}_\kappa$?
\end{question}

We included the possibility of the consistency only, because at least in the
case of $\kappa=\omega_1$ or $\kappa=\cc$ it is consistent that
there is no Banach space of density $\kappa$ which contains isomorphic copies of all
Banach spaces of density $\kappa$ whose dual balls 
with the weak$^*$ topology are uniform Eberlein compact spaces (\cite{christina}) and so
in particular are Corson compacta.
As we noted such Banach spaces, at least for $\kappa=\cc$, are in $\mathcal SI_\cc$. 

\begin{remark}\label{noe} There is no separable indecomposable Banach space which admits 
as subspaces all Banach spaces which are subspaces of separable indecomposable Banach spaces.
This follows from the fact that subspaces of separable indecomposable Banach spaces
include all separable reflexive spaces by \cite{argyros-r} but a separable Banach space $\mathcal X$  which
contains all separable reflexive Banach spaces must contain all separable Banach spaces by \cite{bourgain},
hence $\mathcal X$ must contain $c_0$ and so is not indecomposable.
\end{remark}

\section{Notation and terminology}

All topological spaces in this paper are Hausdorff. In particular, compact spaces or regular spaces are
meant to be Hausdorff.
If $X$ is a topological space, $C(X)$ will denote the vector space 
of real-valued continuous functions on $X$, 
$C_b(X)$ will denote the Banach space of bounded real-valued continuous functions 
on $X$ with the supremum norm, $C_1(X)$ will denote its subset consisting
 of functions whose range is a subset of $[0,1]$, 
 and $C_0(X)$ will be the norm closure of the set of 
 compactly supported functions in $C_b(X)$. 
 If $K$ is compact Hausdorff, we clearly have 
 $C_b(K) = C_0(K) = C(K)$. 
 Each of the sets $C(X)$, $C_b(X)$, 
 $C_0(X)$ and $C_1(X)$ is a lattice when equipped with the 
 pointwise maximum and minimum, 
 i.e. $(f \vee g) (t) = \max(f(t),g(t))$ and $(f \wedge g)(t) = \min(f(t),g(t))$. Moreover,
  the lattice $C_1(X)$ has the least element $0$. If $\A\subseteq C(X)$, then
  $(\A)_1=\A\cap C_1(X)$.
  
  For a real-valued function $f$ we denote its cozero set by $\coz f = f^{-1}[\R \setminus \{0\}]$ 
and its support by $\supp f = \overline{\coz f}$. If $\A$ is a family of functions, 
we set $\coz (\A) = \{\coz f: f \in \A\}$. For a compact Hausdorff space $K$ 
we write $\coz(K) = \coz(C(K))$. Note that $\coz (K)$ is closed under countable unions.
We say that two functions $f_1$ and $f_2$
  with the same domain are disjoint if 
  $\coz f_1 \cap \coz f_2 = \emptyset$. 
  Note that we only require disjointness of the cozero sets,
  not of the supports. It is easy to see that if we have finitely many
   pairwise disjoint non-negative functions $f_1,\dots,f_k$ then $f_1 \vee \cdots \vee f_n = \sum_{i=1}^n f_i$.
  
  If $X$ is a Hausdorff completely regular space (a Tichonov space), 
we write $\beta X$ for the \v Cech-Stone compactification of $X$, 
and $X^* = \beta X \setminus X$ for the \v Cech-Stone remainder of $X$. 
We also denote by $\beta: C_b(X) \rightarrow C(\beta X)$ the \v Cech-Stone 
extension operator which maps every bounded continuous function on $X$ to its 
(uniquely determined) extension on $\beta X$. The extension operator $\beta$ 
is a linear onto isometry and a lattice isomorphism. Indeed, for any $x \in X$ 
and $f,g \in C_b(X)$ we have $(\beta(f \vee g))(x) = (f \vee g)(x) = f(x) \vee g(x) = 
(\beta f)(x) \vee (\beta g)(x)$. By the continuity and by the density the same holds for 
all $x \in \beta X$ and $\beta(f \vee g) = (\beta f) \vee (\beta g)$. 
The other properties are proved analogically. 
For more information about \v Cech-Stone compactification see e.g. Chapter 6 of \cite{gj}.

We consider all algebras to be unital and over the field of real numbers
as we will be applying them in the Banach space theory. In particular, if $\A$ is a subset of $C(K)$ for some compact Hausdorff $K$,
 then the subalgebra of $C(K)$ generated by $\A$ will be always assumed to contain constant functions. 
 Unless specifically mentioned, we do not assume subalgebras of $C(K)$ to be closed  in the norm metric.
  If $\A$ is a subalgebra of $C(K)$ and $f \in C(K)$, we denote by $\A \langle f \rangle$ the
   algebra generated by $\A \cup \{f\}$. Any element of $\A \langle f \rangle$
    is of the form $g_0 + \sum_{i=1}^m g_i \cdot f^i$ for some $g_0,\dots,g_m \in \A$.

The cardinality of a set $A$ is denoted as $|A|$.
If $A$ is a set, we denote by $\wp(A)$ the power set of $A$, that is the set of all subsets of $A$. 
We use $\cc$ to denote the cardinality of continuum, that is the cardinality of $\wp(\N)$.
 If $f: A\rightarrow A'$ is a function
and $B\subseteq A$, then $f|B$ denotes the restriction of $f$ to the set $B$. 
We use both symbols $\omega$ and $\N$ to denote the smallest infinite ordinal. 
As usual we identify $k\in \N$ with the set $\{0, \dots, k-1\}$, so we will
sometimes use $A\cap k$ or $A\setminus k$ for $A\subseteq\N$ and $k\in \N$.

We will work with the sets $\{0,1\}^{<\omega} = \bigcup_{n \in \N} \{0,1\}^n$ and $\{0,1\}^\N$.
Note that if $n\in \N$ and $\sigma\in \{0,1\}^\N$ or $s\in \{0,1\}^{<\omega}$ satisfies $n\leq|s|$ ($|s|$
is the cardinality of $s$ as a set, i.e. the length of the sequence $s$),
then $s|n, \sigma|n\in \{0,1\}^n$.  By $(0)$ and $(1)$ we will denote 
elements of $\{0,1\}^1$ assuming values $0$ or $1$ respectively.
  If $s \in \{0,1\}^{n}$ and $t \in \{0,1\}^{m} $ for $n, m\in \N$, 
  we denote by $s^\frown t\in \{0,1\}^{n+m}$ the concatenation of $s$ and $t$, that is
  $(s^\frown t)(i)=s(i)$ if $i<n$ and $(s^\frown t)(i)=t(i-n)$ if $n\leq i<n+m$.
  In particular we will often use $s^\frown(0)$ or $s^\frown(1)$.

Banach spaces that will be of our interest are the space $C(K)$ of continuous functions on 
a compact Hausdorff space $K$ with the supremum norm, the space $\ell_\infty$ of bounded
 functions on $\N$ with the supremum norm, the space 
 $\ell_1(\kappa) = \{(x_i)_{i \in \kappa} \in \R^\kappa: \sum_{i \in \kappa} |x_i| <\infty\}$ 
 equipped with the norm $\|(x_i)_{i \in \kappa}\| = \sum_{i \in \kappa} |x_i|$ (where $\kappa$ is a cardinal), 
 and the spaces of bounded linear operators between these spaces with the operator norm. 
 We will use $\| \cdot \|$ to denote any norm, the ambient Banach space on which this norm is 
 defined should be understood from the context. If $A$ is a set and $f: A \rightarrow \R$
  is a function we write $\|f\|_\infty = \sup_{a \in A} |f(a)|$. 
  We use this notation even in the cases where $\|\cdot\|_\infty$ is not a norm.
If $\XX$ is a Banach space, we denote by $\XX^*$ its dual. 
The space $\XX^*$ can be equipped with the locally convex weak$^*$ topology.
We will be particularly interested in the dual of $C(K)$, for compact Hausdorff $K$, 
which can be identified with the space $M(K)$ of signed Radon measures on $K$ with the variation norm.
 Otherwise we use standard Banach space terminology and notation as in \cite{fabian-etal}.
 
 All operators from one Banach space to another are meant to be linear and bounded.
 If $T:\XX\rightarrow\Y$ is such an operator for Banach spaces $\XX, \Y$, by
 $T^*:\Y^*\rightarrow\XX^*$ we denote the adjoint of $T$.
 If the Banach spaces are closed algebras of functions we consider
 also multiplicative operators, i.e., the ones which preserve the multiplication of functions.
 
 Note that following standard conventions the symbol $*$ is abused: $X^*$, $\XX^*$ and $T^*$
 denote the \v Cech-Stone compactification of the completely regular topological space $X$, 
 the dual space of the Banach space $\XX$ and
 the adjoint of the operator $T$ respectively. As usual we hope to avoid confusion by using
 different fonts and diferent collections of letters.

\section{Algebras of functions and their Gelfand spaces}

In this section, $K$ will denote a compact Hausdorff space.
 Let us recall that we use the convention that every algebra is assumed to be unital
 and over the field of real numbers.

\subsection{Gelfand spaces and maximal sets}

The aim of this subsection is
to summarize some information about the Gelfand space $\nabla \A$ associated
with a subalgebra $\A\subseteq C(K)$, where $K$ is compact and Hausdorff. This corresponds to a general construction
in the theory of $C^*$-algebras, but we present it in the context of real-valued
algebras $C(K)$ as we will use it in this paper.

\begin{proposition}\label{stone}$ $
\begin{enumerate}
\item  {\rm [4.2.2, 4.2.3 \cite{semadeni}]} If $K, L$ are compact and Hausdorff and $\phi: K\rightarrow L$ is
a continuous surjection, then $T_\phi: C(L)\rightarrow C(K)$ defined
by $T_\phi(f)=f\circ \phi$ for $f\in C(L)$ is a multiplicative isometric linear
embedding.
\item {\rm [M. H. Stone (7.5.2 \cite{semadeni})]} Suppose that $K$ is a compact Hausdorff space and $\A$ is a
closed subalgebra of $C(K)$. Then there is a compact Hausdorff $L$ and
a continuous surjection $\phi: K\rightarrow L$ such that 
$$\A=\{f\circ\phi: f\in C(L)\}.$$
In particular, by (1), the algebras $\A$ and $C(L)$ are isomorphic.
\end{enumerate}
\end{proposition}

\begin{definition} \label{definition-nabla}
Suppose that $K$ is a compact Hausdorff space and $\A$ is a nonempty subfamily of $C(K)$. We define a mapping
$\Pi_\A: K \rightarrow \R^{\A}$ by
$$(\Pi_\A(x))(f)=f(x)$$
for all $f\in \A$ and $x\in K$ and we define $\nabla\A$ to be the image of
$K$ under $\Pi_\A$.
\end{definition}

In the following proposition we point out some properties of the mapping $\Pi_\A$.

\begin{proposition} \label{proposition-Pi}
    Suppose that $K$ is a compact Hausdorff space and $\A$ is a nonempty subset of $C(K)$. Then:
    \begin{enumerate}
        \item The induced composition operator  $T_{\A}: C( \nabla \A) \rightarrow C(K)$, 
        $(T_{\A} (f))(x) = (f\circ\Pi_\A) (x)$ is an isometric multiplicative linear embedding 
        whose image is the closed Banach subalgebra of $C(K)$ generated by $\A$.
        \item If $\A$ is a closed subalgebra of $C(K)$, then $\A$ is multiplicatively
        isometrically isomorphic to $C(\nabla\A)$.
        \item If $\mathcal{B}$ is the closed algebra generated by $\A$, then the canonical
        projection $\pi_{\A, \B}: \R^\B\rightarrow \R^\A$ restricted to $\nabla\B$ is
        a homeomorphism between
        $\nabla \B$ and $\nabla \A$.
    \end{enumerate}
\end{proposition}

\begin{proof}
    The item  (1) is essentially Proposition \ref{stone}  (2) and is shown
     in Proposition 3.11 of \cite{superfew} (cf. 7.5.2. of \cite{semadeni}). The item (2) is
     a consequence of item  (1).
    For (3) note that by the definitions $\pi_{\A, \B}$ restricted to $\nabla\B$ is onto $\nabla\A$ and
     $T_\B(f\circ \pi_{\A, \B})=T_\A(f)$ for
    all $f\in C(\nabla \A)$. 
   On the other hand by (1) the images of
    $T_\A$ and $T_\B$ are the same, so $\pi_{\A, \B}$ restricted to $\nabla\B$ is also
    injective and so a homeomorphism onto $\nabla\A$.
\end{proof}

\begin{lemma}\label{homeo} Suppose that $K, L$ are compact Hausdorff spaces and $\phi: K\rightarrow L$
is a continuous surjection, then $\nabla\A$ is homeomorphic to $L$, where
$\A=\{f\circ \phi: f\in C(L)\}$. In particular $\nabla C(K)$ is homeomorphic to $K$.
\end{lemma}
\begin{proof} By Proposition \ref{stone} (1) the algebra  $C(L)$ is linearly isometric to
$\A$ and by Proposition \ref{proposition-Pi} (2) the algebra $\A$ is linearly isometric to
$C(\nabla\A)$. So, by the Banach-Stone theorem $L$ and $\nabla\A$ are homeomorphic.
\end{proof}

\begin{definition}\label{def-maximal}
Suppose that $K$ is a compact Hausdorff space and 
 $\A$ is a subalgebra of $C(K)$. We say that $D \subseteq K$ is $\A$-maximal if $\Pi_\A^{-1} [\Pi_\A [D]] = D$.
\end{definition}

The following lemma lists some properties of $\A$-maximal sets:

\begin{lemma} \label{lemma-A-maximal}
    Suppose that $K$ is a compact Hausdorff space.  Let $\A$ be a subalgebra of $C(K)$
    \begin{enumerate}
        \item A set $D \subseteq K$ is $\A$-maximal if and only if there is $E \subseteq \nabla \A$ such that $D = \Pi_\A^{-1}[E]$.
        \item Taking preimages under $\Pi_\A$ is a bijection between the family of all subsets of $\nabla \A$
        and the family of $\A$-maximal sets. Its inverse is taking the images under $\Pi_\A$.
        
        \item The family of $\A$-maximal sets is closed under unions, intersections and complements.
        \item An $\A$-maximal set $V$ is open (resp. closed) if and only if $\Pi_\A [V]$ is open (resp. closed).
        \item If $E \subseteq K$ is $\A$-maximal and $D \subseteq K$ is arbitrary, 
        then $$\Pi_\A [E] \cap \Pi_\A[D] = \Pi_\A [E \cap D].$$
        \item If $f \in \A$ and $D \subseteq \R$, then $f^{-1}[D]$ is an $\A$-maximal set.
    \end{enumerate}
\end{lemma}

\begin{proof}
    The point (1) is clear.
    The point (2) follows from (1).
    Point (3) also follows from (2) and from the fact that 
    preimages respect set-theoretic operations of intersection, union and complements.

    For (4) note that for $\A$-maximal $V$ if $\Pi_\A [V]$ is open (resp. closed) then
     $V = \Pi_\A^{-1} [\Pi_\A [V]]$ is open (resp. closed) by continuity of $\Pi_\A$. 
     On the other hand, if $V$ is closed, then $\Pi_\A [V]$ is closed as $\Pi_\A$ 
     is a closed map (as its domain is compact) -- this does not require $V$ to be $\A$-maximal. 
     $\Pi_\A$ is not, however, an open map, 
     so we need to use $\A$-maximality to prove the last part about openness. 
     But a nonclosed subset of $\nabla\A$ cannot have a closed preimage under $\Pi_\A$ by
     the above. So a nonopen subset of $\nabla\A$ cannot have an open preimage under $\Pi_\A$
     because preimages preserve the complements.

    To show (5) we first note that clearly $\Pi_\A [E] \cap \Pi_\A[D] \supseteq \Pi_\A [E \cap D]$. 
    On the other hand, if $x \in \Pi_\A [E] \cap \Pi_\A [D]$, 
    we can find $d \in D$ such that $\Pi_\A (d) = x$. 
    As $E$ is $\A$-maximal, we get by definition that $d \in \Pi_\A^{-1}[\Pi_\A [E]] = E$. 
    Hence, $d \in E \cap D$ and $x \in \Pi_\A [E \cap D]$.
    
    By Proposition \ref{proposition-Pi} (1) for any $f \in \A$ there is $g = T_{\A}^{-1} (f) \in C(\nabla \A)$
     such that $f = g \circ \Pi_\A$. Hence $f^{-1}[D] = \Pi_\A^{-1} [g^{-1} [D]]$
      is an $\A$-maximal set by (1). We have thus proved (6).
\end{proof}

\subsection{Separation of sets by algebras of functions}

\begin{definition} Suppose that  $K$ is a compact Hausdorff space and $\A \subseteq C(K)$ and $D, E$ are subsets of $K$. 
We say that
$\A$ separates $D$ and $E$ if and only if
there is $f\in \A$ such that
$$\overline{f[D]} \cap \overline{f[E]}=\emptyset.$$
Otherwise we say that $\A$ does not separate $D$ and $E$.

\end{definition}

\begin{lemma} \label{lemma-separation-equivalence} Suppose that $K$ is a compact Hausdorff space 
    and $\A\subseteq C(K)$ and $D, E \subseteq K$. The following are equivalent:
    \begin{enumerate}
        \item $\A$ separates $D$ and $E$;
        \item The norm closure $\overline{\A}$ of $\A$ separates $D$ and $E$.
    \end{enumerate}
    If $\A$ is a subalgebra of $C(K)$, then the preceding points are also equivalent to
    \begin{enumerate}
        \item[(3)] The sets $\Pi_\A[D]$ and $\Pi_\A[E]$ have disjoint closures in $\nabla \A$.
    \end{enumerate}
\end{lemma}

\begin{proof}
    The nontrivial implication of the equivalence (1) and (2) 
    follows from the fact that if $f \in \overline{\A}$ is such that $\overline{f[D]} 
    \cap \overline{f[E]}=\emptyset$, then $\overline{f[D]}$ and $\overline{f[E]}$
     are two disjoint compact subsets of $\R$, and hence have positive 
     distance $\varepsilon > 0$. We can then take $f' \in \A$
      such that $\norm{f - f'}< \varepsilon/3$, 
      and then clearly $\overline{f'[D]} \cap \overline{f'[E]} = \emptyset$.
    
    Note that by Proposition \ref{proposition-Pi} (1)  for
     every $f \in  C( \nabla \A)$ the function $T_{\A}(f)$ is an element of $\overline{\A}$ (as $\A$ is an algebra) and
    \begin{align*}
        (T_{\A}(f))[D] &= f \left[\Pi_\A[D] \right], \\
        (T_{\A}(f))[E] &= f \left[\Pi_\A[E] \right].
    \end{align*}
    Hence, by the normality of $\nabla\A$, the sets $\Pi_\A[D]$ and $\Pi_\A[E]$ have 
    disjoint closures in $\nabla \A$ if and only if they are separated by some $f \in C( \nabla \A)$
     if and only if the function $T_{\A} (f) \in \overline{\A}$ 
     witnesses that $\overline{\A}$ separates $D$ and $E$. Hence, we have $(2) \Leftrightarrow (3)$.
\end{proof}

\begin{lemma} \label{piecewise-separation} Suppose that $K$ is a compact Hausdorff space,
   $\A$ is a subalgebra of $C(K)$ and $\V$ is
    a cover of $K$ consisting of open $\A$-maximal sets. Suppose further that $D , E \subseteq K$ and 
     that for every $V \in \V$, $\A$ 
    separates $D \cap V$ and $E \cap V$. Then $\A$ separates $D$ and $E$.
\end{lemma}

\begin{proof}
    Suppose that $\A$ does not separate $D$ and $E$.
     Then by Lemma \ref{lemma-separation-equivalence} (3) 
     we have that there is some $y \in \nabla \A$ such that 
     $y \in \overline{\Pi_\A [D]} \cap \overline{ \Pi_\A [E]} \neq \emptyset$. 
     As $\{ \Pi_\A [V]: V \in \V\}$ is an open cover of $\nabla \A$
     by Lemma \ref{lemma-A-maximal} (4), there is $V \in \V$ such that $y \in \Pi_\A [V]$. 
      Since  $\Pi_\A [V]$ is open, $y\not\in \overline{\Pi_\A [D]\setminus \Pi_\A [V]}$,
      so $y\in \overline{\Pi_\A [D]\cap \Pi_\A [V]}$ as we can apply ${\overline{A\cup B}}=\overline{A}\cup\overline{B}$
      for $A=\Pi_\A [D]\setminus \Pi_\A [V]$ and $B=\Pi_\A [D]\cap \Pi_\A [V]$. 
      Similarly  $y\in \overline{\Pi_\A [E]\cap \Pi_\A [V]}$.
      
      As $V$ is $\A$-maximal, we get by Lemma \ref{lemma-A-maximal} (5) 
      that $y\in \overline{\Pi_\A [D\cap V]}$ and that $y\in \overline{\Pi_\A [E\cap V]}$. Hence 
      by Lemma \ref{lemma-separation-equivalence} (3)
       we get that $\A$ does not separate $D \cap V$ and $E \cap V$ -- a contradiction.    
\end{proof}

\begin{lemma}\label{nonseparation-zero} Suppose that $K$ is a compact Hausdorff space 
    $D, E\subseteq K$ and $\A$ is a subalgebra of $C(K)$ which does 
    not separate $D$ and $E$. Let $f \in C(K)$ be a function such that $f|D = 0$ and $f|E = 0$. 
    Then the algebra $\A \langle f \rangle$ generated by $\A$ and $f$ does not separate $D$ and $E$.
\end{lemma}

\begin{proof}
    Any element $g$ of $A \langle f \rangle$ is of the form 
    $g  = g_0 + \sum_{k=1}^m g_k \cdot f^k$ for some $m \in \N$ and $g_0,\dots,g_m \in \A$.
     It follows from the assumptions that $g[D] = g_0[D]$ and $g[E] = g_0[E]$. 
     Hence, as $g_0$ cannot separate $D$ and $E$, neither can $g$.
\end{proof}

The following lemma is clear.

\begin{lemma}\label{nonseparation-limit} Suppose that $K$ is a compact Hausdorff space 
 $D, E \subseteq K$ and $\gamma$ is an ordinal. 
If $(\A_\alpha)_{\alpha < \gamma}$ is an increasing family 
of subalgebras of $C(K)$, none of which separate $D$ and $E$, then $\bigcup_{\alpha < \gamma} \A_\alpha$ 
does not separate $D$ and $E$.
\end{lemma}

\subsection{Copies of $\beta\N$ in Gelfand spaces}

We will often use the following:

\begin{proposition}[Talagrand \cite{talagrand}]\label{talagrand}
Suppose that $\XX$ is a Banach space and $B_{\XX^*}$ is  its dual unit ball with the
weak$^*$ topology. The following are equivalent:
\begin{enumerate}
    \item $\XX$ admits a subspace isomorphic to $\ell_1(\cc)$;
    \item $\XX$ admits a quotient isomorphic to $\ell_\infty$;
    \item $\XX$ admits a quotient isometric to $\ell_\infty$;
    \item $B_{\XX^*}$ contains a homeomorphic copy of $\beta \N$.
\end{enumerate}
\end{proposition}

\begin{proof}
    The equivalence of (1) and (2) is well known, 
    but we include the short proof for the sake of completeness. 
    Suppose that $\XX$ contains a subspace isomorphic to $\ell_1(\cc)$. 
    Let $Q:\ell_1(\cc) \rightarrow \ell_\infty$ be  a surjective operator
    (which exists as $\ell_1(\cc)$ can be continuously linearly mapped onto
     any Banach space of density at most $\cc$). By the injectivity of $\ell_\infty$
     we can extend $Q$ to a surjective operator $\tilde{Q}:\XX \rightarrow \ell_\infty$, 
     which, by the first isomorphism theorem proves (1) $\Rightarrow$ (2). 
     
     On the other hand suppose that $Q: \XX \rightarrow \ell_\infty$
      is a surjective operator. As $\ell_1(\cc)$ is isomorphic to a subspace of 
      $\ell_\infty$, we can consider the subspace
       $\Y = Q^{-1}[\ell_1(\cc)]$ of $\XX$. 
       Then $Q|\Y:\Y \rightarrow \ell_1(\cc)$ is a surjective operator and we can 
       use the projectivity of $\ell_1(\cc)$ to find a subspace of $\Y$ (and thus a subspace of $\XX$) 
       which is isomorphic to $\ell_1(\cc)$, proving that (2) $\Rightarrow$ (1).
    
    The implications from (4) to (3) and (2) were proved by Talagrand \cite{talagrand}.
    The implication from (2) to (4) follows from the fact that $\ell_\infty$
    is isometric to $C(\beta\N)$ and so any linear surjection $Q:\XX\rightarrow \ell_\infty$
    induces weak$^*$ continuous  and hence bounded $Q^*|\beta\N: \beta\N\rightarrow \XX^*$. 
\end{proof}

\begin{lemma}\label{betaN-embedding} Suppose that a Banach space $\XX$  of density not bigger than $\cc$
does not admit $\ell_\infty$ as a quotient. Then $\XX$ isometrically
embeds into a Banach space $C(K)$, where $K$ is a compact Hausdorff space of  weight at most $\cc$
which does not admit $\beta\N$ as a subspace.
\end{lemma}
\begin{proof} Take $K=B_{\XX^*}$ with the weak$^*$ topology
and use Proposition \ref{talagrand}.
\end{proof}

\begin{lemma}\label{betaN-product} Suppose that  $K$ is a compact Hausdorff space and  $\A, \B\subseteq C(K)$
are subalgebras of $C(K)$ such that neither
$\nabla\A$ nor $\nabla \B$ admits a homeomorphic copy of $\beta\N$ as a subspace.
Then $\nabla \C$ does not admit $\beta\N$ as a subspace,
where $\C$ is the subalgebra of $C(K)$ generated by $\A$ and $\B$.
\end{lemma}
\begin{proof} By Proposition \ref{proposition-Pi} (3)
the natural projections from $\nabla \overline{\C}\subseteq \R^{\overline{\C}}$ onto $\nabla \C\subseteq \R^\C$
or $\nabla(\A\cup\B)\subseteq \R^{\A\cup \B}$ are homeomorphisms, so
$\nabla \C$ and $\nabla(\A\cup\B)$ are homeomorphic.
It is clear that $\nabla(\A\cup \B)$ is a subspace of $\nabla\A\times\nabla\B$.
Hence, if $\nabla\C$ admits a copy of $\beta\N$, then so does $\nabla\A\times\nabla\B$.
However, by a result of \cite{talagrand} (cf. Corollary 3.5 of \cite{negrepontis}) this
means that one of $\nabla\A$ or $\nabla \B$ must admit a copy of $\beta\N$.
\end{proof}

\begin{lemma}\label{projectivity}
Suppose that $K, L$ are compact Hausdorff spaces and $\phi: K\rightarrow L$ is 
a continuous surjection. If $L$ has a subspace homeomorphic to
$\beta\N$, then so does $K$
\end{lemma}
\begin{proof}
This is the projectivity of $\beta\N$ among compact Hausdorff spaces (\cite{gleason}).
\end{proof}

\begin{lemma}\label{small-algebra} Let $K$ be a compact Hausdorff space.
Suppose that $\XX$ is a Banach space 
which does not admit a quotient isomorphic to $\ell_\infty$ and suppose 
that there is a continuous surjection of $K$ onto the
 dual ball $B_{\XX^*}$ with the weak$^*$ topology. Then there is an isometric linear embedding
of $T: \XX \rightarrow C(K)$ such that 
$\nabla\A$ does not admit a copy of $\beta\N$, where $\A$ is the closed subalgebra of $C(K)$ generated by $T[\XX]$.
\end{lemma}

\begin{proof} Let $\phi: K \rightarrow B_{\XX^*}$ be a continuous surjection.
 The induced composition operator  $T_\phi : C(B_{\XX^*}) \rightarrow C(K)$, $T_\phi(f) = f\circ \phi$, 
is a multiplicative linear  isometry onto its image  by Proposition \ref{stone}.
Let $J:\XX\rightarrow C(B_{\XX^*})$ be the canonical embedding. 
Let $T=  T_\phi\circ J$ and $\A$ be the closed subalgebra of $C(K)$ 
generated by $T[\XX]$. As $J[\XX]$ separates points of $B_{\XX^*}$, 
the algebra it generates in $C(B_{\XX^*})$ is dense in $C(B_{\XX^*})$ by the Stone-Weierstrass theorem. 
It follows from the fact that $T_\phi$ 
is a multiplicative linear isometry  
that $\A = T_\phi [C(B_{\XX^*})]$,  and so $C(\nabla\A)$ is
linearly isometric to $C(B_{\XX^*})$ by Proposition \ref{proposition-Pi} (2).
Hence by the Banach-Stone theorem  $B_{\XX^*}$ must be homeomorphic
 to $\nabla \A$. 
 By Proposition \ref{talagrand}, $B_{\XX^*}$
  does not contain a homeomorphic copy of $\beta \N$, and hence neither does $\nabla \A$.
\end{proof}

\begin{lemma}\label{new-promise}
Suppose  that $K$ is a compact Hausdorff space, $\A$ is a subalgebra of $C(K)$ and
$\{V_n: n\in \N\} \subseteq \coz(\A)$ are pairwise disjoint and nonempty.
If $\nabla\A$ does not contain a copy of $\beta\N$, then there is an infinite $B\subseteq\N$
and a countable $D\subseteq \bigcup_{n \in B} V_n$ and a countable
$E\subseteq \bigcup_{n \in \N\setminus B} V_n$ such that $D$ and  $E$ are not separated by $\A$ and
$|D\cap V_n|=|E\cap V_{n'}|=1$ for every $n\in B$ and every $n'\in \N\setminus B$.
\end{lemma}
\begin{proof}
For $n\in \N$ let $f_n\in C(\nabla\A)$ be such that
$(f_n\circ \Pi_\A )^{-1}[\R\setminus\{0\}]=V_n$.  The existence of these objects follows from Proposition \ref{proposition-Pi} 
(1).
Let $W_n\subseteq \nabla\A$ be given by $f_n^{-1}[\R\setminus\{0\}]$.
It follows that $W_n$'s are pairwise disjoint as well.
Pick for all $n\in \N$ any element $x_n$ of $V_n$.
Consider $\phi: \N\rightarrow\nabla\A$ given by $\phi(n)=\Pi_\A (x_n)\in W_n$.
It follows that $\phi(n)$'s are all distinct for $n \in \N$. Consider the continuous extension
$\beta\phi:\beta\N\rightarrow\nabla\A$ of $\phi$. By the hypothesis of the lemma $\beta\phi$
cannot be a homeomorphism (equivalently injective) on any copy of $\beta\N$.
As $\N^*$ contains copies of $\beta\N$, we can find distinct $u, v\in \N^*$ such that
$\beta\phi(u)=\beta\phi(v)$. Let $B\subseteq \N$ be such that
$u\in {\overline{\{n: n\in B\}}}$ and $v\in {\overline{\{n: n\in \N\setminus B\}}}$, where the closures are
taken in $\beta\N$.
It follows that ${\overline{\{\Pi_\A(x_n): n\in B\}}}$ intersects ${\overline{\{\Pi_\A(x_n): n\in \N\setminus B\}}}$.
So $D=\{x_n: n\in B\}$  and $E=\{x_n: n\in \N\setminus B\}$ 
satisfy the conclusion of the lemma by Lemma \ref{lemma-separation-equivalence} (3).
\end{proof}

\section{\v Cech-Stone reminders and submorphisms}

\subsection{\v Cech-Stone reminders}

Let us recall that for every completely regular topological space $X$ 
we denote by $\beta X$ its \v Cech-Stone compactification and 
by $X^*$ its \v Cech-Stone remainder $X^* = \beta X \setminus X$.

\begin{lemma} \label{weight-remainder}
    Let $X$ be a $\sigma$-compact Hausdorff space of weight at most $\cc$. 
    Then the cardinality of $C(X)$ and the weight of $X^*$ is at most $\cc$.
\end{lemma}

\begin{proof}
Note that in general if $Y$ is a compact Hausdorff space of weight not bigger than
an infinite cardinal $\kappa$, then $|C(Y)|\leq \cc\kappa^\omega$. This is because
using the open basis of cardinality not bigger than  $\kappa$
one can construct a family of the same cardinality in $C(Y)$ which separate points of $Y$.
So by the Stone-Weierstrass theorem the closed algebra generated by it is $C(Y)$ and
clearly such an algebra has cardinality not bigger than $\cc\kappa^\omega$. 

Let $X=\bigcup_{n\in \N}X_n$, where $X_n$ is compact. The above observation
gives us that $|C(X)|\leq |\Pi_{n\in\N}C(X_n)|\leq (\cc^\omega)^\omega=\cc$.
It follows that $|C(\beta X)|\leq \cc$ and so the weight of $\beta X$ is not bigger than $\cc$.
Hence, the weight of its subspace $X^*$ is not bigger than $\cc$.

\end{proof}

\begin{lemma}\label{closure-remainder} 
Suppose that $X$ is completely regular Hausdorff and    $U\subseteq\beta X$ is open. Then
$U\subseteq \overline{U\cap \X}$, where the closure is taken in $\beta\X$.
\end{lemma}
\begin{proof}
Consider $V=U\setminus \overline{U\cap \X}$ which is an open subset of $U$
disjoint from $U\cap \X$ and  hence disjoint from $X$. As 
$X$ is dense in $\beta\X$ we have $V=\emptyset$, as required.
\end{proof}

\begin{definition}[cf. 14N4 of \cite{gj}]
    A compact Hausdorff space $X$ is an $F$-space if for every  disjoint  open
    $F_\sigma$ sets $F,G$ in $X$ we have $\overline{F} \cap \overline{G} = \emptyset$.
\end{definition}

\begin{lemma}\label{X*-Fspace} If $X$ is a locally compact, $\sigma$-compact Hausdorff space,
 then $X^*$ is a compact F-space.
\end{lemma}

\begin{proof}
This is Theorem 14.27 of \cite{gj}. 
\end{proof}

We note that the \v Cech-Stone extension operator $\beta : C_b(X) \rightarrow C(\beta X)$
 that maps a bounded continuous function on $X$ to its
  \v Cech-Stone extension is an onto isometry and both a lattice and algebra isomorphism. 
  We also remark that if $f \in C_0(X)$ then $(\beta f)|X^* = 0$.

\subsection{Algebras of functions on almost $P$-spaces}

\begin{definition}\label{def-pspace-realc} A topological space is called an almost $P$-space if
all nonempty $G_\delta$ subsets of it have nonempty interior. A
topological space is realcompact if it can be homeomorphically embedded as
a closed subspace of a product of real lines (cf. 11.12 of \cite{gj}). 
\end{definition}

For more information on almost $P$-spaces see \cite{levy}.

\begin{lemma} \label{remainder-P-space}
    Let $X$ be a locally compact, $\sigma$-compact noncompact Hausdorff space. Then $X^*$ is an almost $P$-space
\end{lemma}

\begin{proof}
First we note that any such $X$ is realcompact because it is Lindel\"of and Lindel\"of spaces 
are realcompact by 8.2 of \cite{gj}.
Let $Y\subseteq X^*$ be a nonempty $G_\delta$ set i.e., there are open 
$U_n\subseteq X^*$ such that $Y=\bigcap_{n\in \N} U_n$. Let $y\in Y$ and $f_n\in C_1(X^*)$
be such that $f_n|(X^* \setminus U_n)=1$ and $f_n(y)=0$ for every $n\in \N$. Then 
the zero set of $\sum_{n\in \N} 2^{-n}f_n$ is included in $Y$ and is nonempty as it contains $y$.
Lemma 3.1 of \cite{fineg} says that if $X$ is realcompact and locally compact, then any zero set in $\beta X\setminus X$
is the closure of its interior. In particular, it has a nonempty interior if it is nonempty. 
So, $Y$ has a nonempty interior and we are done.
\end{proof}

\begin{lemma}\label{dense-P}
Suppose that $K$ is a compact Hausdorff almost $P$-space with
an open basis $\B$ and suppose that $\A\subseteq C(K)$
is a closed subalgebra of $C(K)$ such that for every $B\in \B$ there is
$f\in \A\setminus\{0\}$ such that $\coz f \subseteq B$. Then $\nabla\A$ is an almost $P$-space.
\end{lemma}
\begin{proof}
Let $F=\bigcap_{n\in \N}U_n\subseteq\nabla\A$ be a nonempty $G_\delta$ set and $x\in F$, where
$U_n$'s are open. Consider $y\in K$ and $V_n\subseteq K$ such that $\Pi_\A(y)=x$
and $V_n=\Pi_\A^{-1}[U_n]$. The set $\bigcap_{n\in \N} V_n$ is a nonempty $G_\delta$ 
set in the almost $P$-space $K$, and hence
has nonempty interior. Let $B\in \B$ be such that 
$B \subseteq \operatorname{int}(\bigcap_{n \in \N} V_n)$ and let $f\in \A\setminus\{0\}$ be such that $\coz f \subseteq B$.
Let $g\in C(\nabla\A)$ be such that  $T_{\A} (g)=f$ (that is $f=g\circ \Pi_\A$) which exists by
Proposition \ref{proposition-Pi} (1). If $g(x')\not=0$ for some 
$x'=\Pi_\A(y')$ for $y'\in K$, then $y'\in V_n$ for each $n\in \N$, and so
$x'\in F$, and hence $\coz g$ is included in $F$ which proves
that $F$ has nonempty interior as needed to check that $K$ is an almost $P$-space.
\end{proof}

\begin{lemma}\label{disjoint-P} Suppose that $K$ is
a compact Hausdorff almost $P$-space with no isolated points and $V\subseteq K$ is open and nonempty.
Then $V$ contains a pairwise disjoint family of cardinality $\cc$ of open nonempty sets. 
\end{lemma}
\begin{proof}
Since $K$ has no isolated points and is compact (and so regular), every open set contains
 two open subsets whose closures are disjoint.
This allows us to built a family of nonempty open subsets $\{U_s: s\in \{0,1\}^{<\omega}\}$ of $V$
such that  $U_\emptyset=V$, $\overline{U_{s^\frown(0)}}, \overline{U_{s^\frown(1)}}\subseteq U_s$
and $\overline{U_{s^\frown(0)}}\cap\overline{U_{s^\frown(1)}}=\emptyset$ for every  $s\in \{0,1\}^{<\omega}$.
Using the  compactness of $K$ we have the existence of some
 $x_\sigma \in \bigcap_{n\in \N}\overline{U_{\sigma|n}}\subseteq  \bigcap_{n\in \N}U_{\sigma|n}=H_\sigma$
for any $\sigma\in \{0,1\}^\N$.
Since $K$ is an almost $P$-space, this means that the $G_\delta$ sets $H_\sigma$ have nonempty interiors.
Clearly they must be pairwise disjoint, since for every distinct $\sigma, \sigma'\in  \{0,1\}^\N$
there is $n\in \N$ such that $\sigma|n=\sigma'|n$ and $\sigma|(n+1)=\sigma|n^\frown (i)$
 and $\sigma'|(n+1)=\sigma|n^\frown (1-i)$
for some $i\in \{0,1\}$. Then $H_\sigma\subseteq U_{\sigma|n^\frown (i)}$ and 
$H_{\sigma'}\subseteq U_{\sigma|n^\frown (1-i)}$
however $\overline{U_{s^\frown(0)}}\cap\overline{U_{s^\frown(1)}}=\emptyset$ for every  $s\in \{0,1\}^{<\omega}$.

\end{proof}

\subsection{Cantor systems of functions and abundant families}

\begin{definition} \label{cantor-system} Let $\X$ be a locally compact, $\sigma$-compact Hausdorff topological space.
    A system $(f_s)_{s \in \{0,1\}^{< \omega}}$ of functions in $C_1(\X)$ is called a Cantor system if
    \begin{enumerate}
        \item $\supp f_s$ is a compact subset of $\X$ for any $s \in \{0,1\}^{< \omega}$,
        \item $(f_s)_{s \in \{0,1\}^{< \omega}}$ is pairwise disjoint,
        \item for every $x\in X$ there is an open $x\in U\subseteq X$ such that  the set 
        $$\{s \in \{0,1\}^{< \omega}: \supp f_s \cap U \neq \emptyset\}$$ is finite.
    \end{enumerate}
    For a Cantor system $(f_s)_{s \in \{0,1\}^{< \omega}}$ and $\sigma \in \{0,1\}^{\N}$ we define
    \begin{align*}
        f_\sigma = \sum_{n \in \N} f_{\sigma|n}.
    \end{align*}
\end{definition}

Note that the functions $f_\sigma$'s are well defined elements of $C_1(\X)$ 
as the sum in their definition is locally finite by (3) 
and for every $x \in \X$ there is at most one $n \in \N$
such that $f_{\sigma|n}(x) \neq 0$ by (2).
In the next lemma we show that $(\beta f_\sigma) | \X^*$'s are moreover disjoint.

\begin{lemma} \label{cantor-disjoint} Let $\X$ be a locally compact, $\sigma$-compact Hausdorff topological space.
    Let $(f_s)_{s \in \{0,1\}^{<\omega}}$ be a Cantor system
     of functions in $C_1(\X)$. Then $((\beta f_\sigma) | \X^*)_{\sigma \in \{0,1\}^{\N}}$
      are pairwise disjoint in $\X^*$.
\end{lemma}

\begin{proof}
    Let $\sigma_1,\sigma_2$ be two distinct elements of $\{0,1\}^{\N}$. 
    Let $n_0 \in \N$ be the least integer such that $\sigma_1(n_0) \neq \sigma_2(n_0)$.
     For $i = 1,2$ denote  $g_i = \sum_{n>n_0}^{\infty} f_{\sigma_i | n}$. 
     Then $(\beta f_{\sigma_i}) | \X^* = (\beta g_i) | \X^*$
      as $f_{\sigma_i} - g_i = \sum_{n=0}^{n_0} f_{\sigma_i | n} \in C_0(\X)$ by property (1) of a Cantor system. 
      Further, $g_1$ and $g_2$ are disjoint in $\X$ by property (2) of a Cantor system. 
      Let $U = \{x \in \beta \X: (\beta g_1)(x) \neq 0 \text{ and } (\beta g_2)(x) \neq 0\}$. 
      Then $U$ is open, and thus $U  \subseteq \overline{U \cap \X}$ by Lemma \ref{closure-remainder}.
       But $U \cap \X$ is empty as $g_1$ and $g_2$ are disjoint in $\X$, 
       and thus $U$ is empty as well. Hence, $\beta g_i$, $i=1,2$, 
       are disjoint in $\beta \X$ and $(\beta f_{\sigma_i}) | \X^* = (\beta g_i) | \X^*$, $i = 1,2$,
        are disjoint in $\X^*$ as required.    
\end{proof}

\begin{definition} \label{def-abundant}
Let $X$ be a locally compact, $\sigma$-compact Hausdorff topological space.
A family of functions $\F\subseteq C(\X^*)$ is called abundant if the following holds:
whenever $(f_s)_{s \in \{0,1\}^{< \omega}}$ 
is a Cantor system of functions in $C_1(\X)$, then there are distinct
 $\sigma_k \in \{0,1\}^\N$, $k \in \N$, such that $(\beta f_{\sigma_k})|\X^* \in \F$ for each $k \in \N$.
\end{definition}

\begin{lemma}\label{abundant-equi} Let $X$ be a locally compact, $\sigma$-compact Hausdorff topological space.
A family of functions $\F\subseteq C(\X^*)$ is  abundant if and only if the following holds:
whenever $(f_s)_{s \in \{0,1\}^{< \omega}}$ is a Cantor 
    system of functions in $C_1(\X)$, then there is $\sigma \in \{0,1\}^\N$ such that $(\beta f_\sigma)|\X^* \in \F$.
\end{lemma}

\begin{proof}
    Clearly the forward implication holds. Now suppose that $\F$ satisfies the condition from the lemma
     and let $(f_s)_{s \in \{0,1\}^{<\omega}}$ 
     be a Cantor system of functions in $C_1(\X)$. 
     We will recursively build a sequence $(\sigma_n)_{n \in \N}$ of distinct 
     elements of $\{0,1\}^\N$ such that $(\beta f_{\sigma_n})|\X^* \in \F$ for every $n \in \N$. 
     First we find $\sigma_0$ by applying the hypothesis to the 
     Cantor system $(f_s)_{s \in \{0,1\}^{<\omega}}$. 
     Now suppose $n \geq 1$ and we have already found $(\sigma_k)_{k < n}$. 
     Let $t \in \{0,1\}^{<\omega}$ be such that it is not an initial segment of 
     any $\sigma_k$ for $k < n$. We define a new Cantor system $(g_s)_{s \in \{0,1\}^{< \omega}}$ by setting
    \begin{align*}
        g_\emptyset &= \sum_{k=0}^{|t|} f_{t|k}, \\
        g_s &=  f_{t^\frown s}, \hspace{1cm} s \in \{0,1\}^{< \omega} \setminus \{\emptyset\}.
    \end{align*}
    That $(g_s)_{s \in \{0,1\}^{< \omega}}$ is a Cantor system follows 
    from the fact that $(f_s)_{s \in \{0,1\}^{< \omega}}$ is. 
    We now apply the hypothesis to $(g_s)_{s \in \{0,1\}^{< \omega}}$ 
    to get $\sigma'$ such that $(\beta g_{\sigma'})|\X^* \in \F$. 
    Set $\sigma_n$ to be the sequence created by $t$ followed by $\sigma'$
     and note that $g_{\sigma'} = f_{\sigma_n}$, 
    and hence $(\beta f_{\sigma_n})|\X^* \in \F$. Moreover, $\sigma_n \neq \sigma_k$ 
    for any $k < n$ as $t$ is an initial segment of $\sigma_n$ but is not an initial 
    segment of any $\sigma_k$, $k < n$. This finishes the recursive construction and the proof of the lemma.
\end{proof}

\begin{lemma}\label{abundant-zero} Let $X$ be a locally compact, $\sigma$-compact noncompact Hausdorff topological space.
Suppose that $\F\subseteq C(\X^*)$ is abundant and that $F_1, \dots, F_j\subseteq \X^*$ for
some $j\in \N$ are closed pairwise disjoint sets such that $f|F_i$ is constant for
every $f\in \F$ and $i=1,\dots, j$.  Suppose that $(f_s)_{s \in \{0,1\}^{< \omega}}$ is a Cantor system
of functions in $C_1(X)$. Then there is $\sigma\in \{0,1\}^\N$ such that
$$(\beta f_\sigma)|(F_1\cup \dots \cup F_j)=0.$$
\end{lemma}
\begin{proof}
Let $\{\sigma_k: k\in \N\}$ be such that $f_{\sigma_k}\in\F$ for every $k\in \N$.
As $(\beta f_{\sigma_k})|\X^*$ are pairwise disjoint by Lemma \ref{cantor-disjoint}
and constant on each $F_i$ by the hypothesis, for each $1\leq i\leq j$ there is at most one $k\in \N$
such that $(\beta f_{\sigma_k})|F_i$ is nonzero. So choose $k\in \N$ such that
 $(\beta f_{\sigma_k})|F_i=0$ for every $1\leq i\leq j$.
\end{proof}
  
  \subsection{Submorphisms}

\begin{definition}
    Let $\mathcal{L}$ and $\mathcal{L}'$ be two lattices with the least elements $0_\mathcal{L}$ and $0_{\mathcal{L}'}$.
    We say that a map $\phi: \mathcal{L} \rightarrow \mathcal{L}'$
     is a submorphism if it is a lattice homomorphism and $\phi(0_{\mathcal{L}}) = 0_{\mathcal{L}'}$.
\end{definition}

By a sublattice we mean a suborder of a bigger lattice which is a lattice and where the joints and the meets are
the same as in the bigger lattice. i.e., the inclusion is a lattice homomorphism.
Lattices with the least elements we will be working with are sublattices of $\wp(A)$  for
 some infinite $A \subseteq \N$ that 
 contain all finite subsets of $A$, and sublattices of $C_1(X)$ where $X$ is a topological space. 
 While the lattices $\wp(A)$ are complete, the lattices $C_1(X)$ we will be working with
 are not complete as $X$ will not be extremally disconnected. 
 
 \begin{lemma}\label{sublattice} Suppose that  $K$ is compact and Hausdorff  and
 $\phi: \wp(\N)\rightarrow C_1(K)$ is a submorphism and that $\XX\subseteq C(K)$ is
 a closed subalgebra of $C(K)$. Then $\phi^{-1}[\XX]$ is a sublattice of $\wp(\N)$ containing $\emptyset$.
 \end{lemma}
 \begin{proof} Clearly $0\in \XX$ and so $\emptyset \in\phi^{-1}[\XX]$. 
Using   Proposition \ref{stone} (2) and the fact that suprema and infima in 
spaces of continuous functions are defined as pointwise maxima and minima
 we conclude that $\XX$ is a sublattice of $C(K)$. 
 Since $\phi$ is a lattice homomorphism it follows that $\phi^{-1}[\XX]$ is a sublattice of $\wp(\N)$.
 \end{proof}
 
 Recall that $a,b \in \mathcal{L}$ are disjoint if $a \wedge b = 0_{\mathcal{L}}$, this agrees 
 with the definition of disjointness of functions in $C_1(X)$ given before -- for $f_1,f_2 \in C_1(X)$
  we have that $\coz(f_1) \cap \coz(f_2) = \emptyset$ if and only if $f_1 \wedge f_2 = 0$.
   Further, if $f_1,\dots,f_k \in C_1(X)$ are disjoint, then $\bigvee_{i=1}^k f_i = \sum_{i=1}^k f_i$.
   We have the following.
   
   \begin{lemma}\label{disjoint-preserv} Let $\mathcal{L}$ and $\mathcal{L}'$ be two 
   lattices with the least elements $0_\mathcal{L}$ and $0_{\mathcal{L}'}$ respectively.
    Any submorphism $\phi: \mathcal{L} \rightarrow \mathcal{L}'$ preserves disjointness.
   \end{lemma}
   
   \begin{lemma}\label{morphism-nice}
Suppose that $X$ is a topological Hausdorff space and and $\phi: \wp(\N) \rightarrow C_1(X)$ is a submorphism
such and $f_n=\phi(\{n\})$ for $n\in \N$.
Then whenever  $A\subseteq \N$ and $x\in X$ is such that $f_n(x)\not=0$ for $n\in A$, then 
$$\phi(A)(x)=f_n(x).$$
\end{lemma}

\begin{proof}
We have
$$\phi(A\setminus\{n\})\wedge\phi(\{n\})=\phi((A\setminus\{n\})\wedge\{n\})=\phi(\emptyset)=0.$$

This means that $\phi(A\setminus\{n\})$ and 
$f_n$ are disjoint, so $\phi(A\setminus\{n\})(x)=0$.
Since
$$\phi(A)=\phi((A\setminus\{n\})\vee(\{n\}))=\phi((A\setminus\{n\})\vee\phi(\{n\})$$
we conclude that $\phi(A)(x)=\phi(\{n\})(x)$ as $\phi(\{n\})$ is non-negative.
\end{proof}

\begin{lemma} \label{lemma-ADsubmorphism}
Suppose that $X$ is a topological Hausdorff space, $(f_n)_{n\in \N}\subseteq C_1(X)$
 and $\phi: \wp(\N)\rightarrow C_1(X)$ is a submorphism such that $\phi(\{n\})=f_n$
for every $n\in \N$. 
If $A, B \in \wp (\N)$ are such that $A\cap B$ is finite then 
$$\phi(A) \wedge \phi(B)=\sum_{n\in A\cap B}f_n.$$
In particular,
$$\phi(A)|Y=\phi(B)|Y=\Bigg(\sum_{n\in A\cap B}f_n\Bigg)|Y,$$
where $Y=\coz(\phi(A))\cap \coz(\phi(B))$.
\end{lemma}

\begin{proof} As $A\cap B$ is the finite disjoint union of singletons of $A\cap B$, we have
\begin{align*}
    \phi(A) \wedge \phi(B) &= \phi(A\cap B) = 
    \phi \left( \bigcup_{n\in A\cap B} \{n\} \right) = \bigvee_{n \in A \cap B} \phi(\{n\}) \\
    &= \bigvee_{n \in A \cap B} f_n = \sum_{n\in A\cap B}f_n.
\end{align*}
The last equality above follows  the fact that $f_n$'s are pairwise disjoint  in $C_1(X)$
by Lemma \ref{disjoint-preserv} as $\phi(\{n\})=f_n$ for $n\in \N$.

If $y\in Y$, then $(\phi(A) \wedge \phi(B))(y)>0$, since the infima in $C_1(X)$ are pointwise,
so by the first part of the lemma
$$0<\Bigg(\sum_{n\in A\cap B}f_n\Bigg)(y)=f_{n_y}(y)\leqno (*)$$
 for some unique $n_y\in A\cap B$ (as
$f_n$'s are disjoint). Hence Lemma \ref{morphism-nice}
implies that 
$$\phi(A)(y)=f_{n_y}(y)$$
which gives the second part of the lemma for $A$ by ($*$). The argument for $B$ is the same. 
\end{proof}

\begin{lemma}\label{morphism-existence}
    Let $X$ be a locally compact, $\sigma$-compact non-compact Hausdorff space.
    Let $(f_n)_{n\in\N} \subseteq C_1(X^*)$ be a pairwise disjoint sequence. 
    Then there exists a submorphism $\phi:\wp (\N) \rightarrow C_1(X^*)$ such that $\phi(\{n\}) = f_n$ for every $n \in \N$.
\end{lemma}

\begin{proof}
    First we extend each $f_n$ to $\tilde{f}_n \in C_1(\beta X)$ such that
    \begin{enumerate}[(a)]
        \item $(\tilde{f}_n)_{n\in\N}$ is pairwise disjoint;
        \item Every compact subset of $X$ intersects at most finitely many supports of $\tilde{f}_n$'s.
    \end{enumerate}
    To see that this could be done let, for $n \in \N$, $F_n = \coz 
    f_n$ and $G_n = \bigcup_{m \neq n} \coz f_m$. 
    These are disjoint cozero sets, and so disjoint open $F_\sigma$ 
    sets in the F-space $X^*$ (see Lemma \ref{X*-Fspace}), and thus 
    $\overline{F_n} \cap \overline{G_n} = \emptyset$. As $X$ is locally 
    compact, $X^*$ is closed in $\beta X$, and these closures are disjoint in $\beta X$ as
    well. We can use its normality to find open disjoint sets $U_n, V_n\subseteq\beta X$ 
    such that $\overline{F_n} \subseteq U_n$ and $\overline{G_n} \subseteq V_n$. 
      
    Let $(X_n)_{n \in \N}$ be an increasing collection of compact subsets of $X$ 
    such that $X = \bigcup_{n \in \N} X_n$. Using the local compactness of $X$ we may assume that
    $X_n=\overline{W_n}$, where $W_n$ are open and increasing and cover $X$. This guarantees that
    every compact subset of $X$ is included in one of the $X_n$'s.
    
    By the Tietze extension theorem for every $n\in\N$ we find $\tilde{f}_n\in  C_1(\beta X)$ 
    such that $\tilde{f}_n | X^* = f_n$ and 
    $\coz \tilde{f}_n \subseteq U_n \cap \bigcap_{m<n} 
    V_m \cap (\beta X \setminus X_n)$. Then (a) holds as 
    $\coz \tilde{f}_n \cap \coz \tilde{f}_m \subseteq U_n \cap V_n = \emptyset$
    for $n < m$, and (b) holds as $\coz \tilde{f}_n \cap X_m$
    is nonempty only if $n < m$.

    For $A \in \wp (\N)$ define $g_A = \sum_{n \in A} \tilde{f_n}| X$. 
    Then $g_A$ is a well-defined continuous function on $X$ 
    as the sum in its definition is locally finite by (b) and the fact that $X$ is locally compact.
    Further, $g_A \in C_1(X)$ by (a). Set $\phi(A) = (\beta g_A)|X^*$. Then for $n \in \N$
    \begin{align*}
        \phi(\{n\}) = (\beta g_{\{n\}})|X^* = (\beta (\tilde{f}_n | X)) | X^* = \tilde{f}_n | X^* = f_n.
    \end{align*}
    
    It remains to prove that $\phi$ is a submorphism.
    It preserves the least element as $g_\emptyset=0$.
    As $\phi$ is the composition of the maps 
    $g_{({\cdot})}: A \mapsto g_A$, 
    $\beta: C_b(X) \rightarrow C(\beta X)$
    and the restriction map from $\beta X$ to $X^*$, 
    we are done if we show that each of these maps preserve suprema and infima. 
     
   Note that for $x\in X$ we have 
    $$\beta(f\wedge g)(x)=(f\wedge g)(x)=
    \min(f(x), g(x))= (\beta(f)\wedge \beta(g))(x)$$
    and similarly for $\vee$, so  the density of $X$ in $\beta X$ 
    implies that the extension operator $\beta$  preserves suprema and infima.
   The restriction map preserves suprema and infima because they are defined pointwise
   as maxima and minima of the values.
     So  we need to show is that $g_{({\cdot})}$ does as well. 
    But this follows from (a): for any $x \in X$ there is $n_x \in \N$
    such that $\tilde{f}_n(x) = 0$ for every $n \in \N \setminus \{n_x\}$
    (it could, of course, happen that $\tilde{f}_{n_x}(x) = 0$ as well). 
    Thus for any $C \in \wp (\N)$ and $x \in X$ we have $g_C(x) = \chi_C(n_x) \tilde{f}_{n_x}(x)$. 
    Hence, for any $A,B \in \wp (\N)$ and any $x \in X$, using the observation above and the
    non-negativity of $\tilde{f}_n$'s for we get

    \begin{align*}
        g_{A \cap B} (x) &= \chi_{A \cap B}(n_x) \tilde{f}_{n_x}(x) = (\chi_A(n_x) \wedge \chi_B(n_x)) \tilde{f}_{n_x}(x) \\
        &= (\chi_A(n_x) \tilde{f}_{n_x}(x)) \wedge (\chi_B(n_x) \tilde{f}_{n_x}(x)) = g_A(x) \wedge g_B(x), \\
        g_{A \cup B} (x) &= \chi_{A \cup B}(n_x) \tilde{f}_{n_x}(x) = (\chi_A(n_x) \vee \chi_B(n_x)) \tilde{f}_{n_x}(x) \\
        &= (\chi_A(n_x) \tilde{f}_{n_x}(x)) \vee (\chi_B(n_x) \tilde{f}_{n_x}(x)) = g_A(x) \vee g_B(x).
    \end{align*}
    Since the suprema and the infima of elements of pairs of elements of $C_b(X)$ are defined as pointwise
    suprema and infima this completes the proof.
\end{proof}

\begin{lemma}\label{morphism-monotone}
Suppose that $X$ is a topological Hausdorff space, $\mathcal P$ is sublattice
of $\wp(\N)$ which contains all finite sets and $\phi:\mathcal P\rightarrow C_1(X)$ is a submorphism.
Then $\phi$ is monotone, i.e., if $A\subseteq A' \in \mathcal{P}$ then $\phi(A)\leq \phi(A')$.
In particular, if $\N \in \mathcal{P}$ and $\phi(\N)|V=0$ for some $V\subseteq X$,
 then $\phi(A)|V=0$ for every $A \in \mathcal{P}$.
\end{lemma}

\begin{proof}
    If $A\subseteq A' \in \mathcal{P}$ then $A \wedge A' = A$, 
    and thus $\phi(A) = \phi(A) \wedge \phi(A') \leq \phi(A')$.
\end{proof}

\begin{lemma}\label{morphism-disjoint}
    Let $X$ be locally compact, $\sigma$-compact non-compact Hausdorff space. 
    Let $(f_n)_{n\in\N} \subseteq C_1(X^*)$ be  pairwise disjoint and $(V_n)_{n \in \N} \subseteq \coz(X^*)$ 
    be such that $f_n|V_m=0$ for every $n, m\in \N$. 
    Then there exists a submorphism $\phi:\wp (\N) \rightarrow C_1(X^*)$ such that $\phi(\{n\}) = f_n$ for every $n \in \N$
    and $\phi(\N)|V_m=0$ for every $m\in \N$.
\end{lemma}

\begin{proof}
    Let $V = \bigcup_{m \in \N} V_m$. Then $V \in \coz(X^*)$. Let $f \in C_1(X^*)$ be such that $V = \coz(f)$.
     Applying Lemma \ref{morphism-existence} to the sequence $(f,f_0,f_1,\dots)$, 
     which is pairwise disjoint by the hypothesis, 
     gives us a submorphism $\phi: \wp(\N) \rightarrow C_1(X^*)$
      such that $\phi(\{n\}) = f_{n-1}$ for $n > 0$ and $\phi(\{0\}) = f$. 
      As $\phi(\{0\})$ and $\phi(\{n > 0\})$ are disjoint by Lemma \ref{disjoint-preserv}, 
      we get that $\phi(\{n > 0\})| (\coz (f)) = 0$. 
 Hence, by relabeling $\phi|\wp(\{n>0\})$ so that it domain is $\wp(\N)$, we are done.
\end{proof}

\subsection{A strengthening of Haydon's lemma}

In this subsection we prove a version of  Lemma 1D of \cite{haydon}.
The main difference, besides moving from the Boolean context to function algebras,
 is that the  algebra $\A$ of our Lemma \ref{haydon} may have density $\cc$,
while the Boolean algebra of Lemma 1D of \cite{haydon} needs to have cardinality smaller than $\cc$.

\begin{lemma}\label{haydon}  Let $K$ be  compact Hausdorff space. 
Suppose that $\A$ is a subalgebra of $C(K)$ which contains
pairwise disjoint $f_n\in C_1(K)$ for $n\in \N$ and $\phi:\wp(\N)\rightarrow C(K)$
is a submorphism such that $\phi(\{n\})=f_n$ for each $n\in \N$. Suppose that $|\I|<\cc$ 
and $D_i$, $E_i\subseteq K$  for $i\in \I$.
Moreover assume that $\A$ does not separate $D_i$ and $E_i$ for any $i\in \I$.

Then there is an infinite $A\subseteq \N$ such that for every $B\subseteq A$, the algebra $\A\langle\phi(B)\rangle$
 generated by $\A$ and $\phi(B)$ does not
separate $D_i$ and $E_i$ for any $i\in \I$.
\end{lemma}
\begin{proof} Suppose that the lemma is false and let us aim at a contradiction.
Let $\{A_\xi: \xi<\cc\}$ be an almost disjoint family of infinite subsets of $\N$ i.e., such that
the intersection of any two distinct elements of it is finite.
As the lemma is assumed to be false, for every $\xi<\cc$ there are (necessarily) infinite $B_\xi\subseteq A_\xi$
such that each algebra $\A\langle\phi(B_\xi)\rangle$  separates $D_{i_\xi}$ and $E_{i_\xi}$ for some $i_\xi\in \I$.
Since $\I$ has cardinality smaller than $\cc$ there is $i\in \I$ such that $i_\xi=i$ 
for uncountably many values of $\xi<\cc$. We fix such $i\in \I$ from this point on.

For all these uncountably many values $\xi<\cc$ we have
disjoint $U_1^\xi, U_2^\xi\subseteq\R$ which are finite unions of bounded open intervals with rational end-points
and we have $\underbar{h}_\xi\in \A\langle\phi(B_\xi)\rangle$ and $m_\xi\in \N$ such that
$$\underbar{h}_\xi= \underbar{g}_{\xi, 0} + \sum_{1 \leq k\leq m_\xi}\underbar{g}_{\xi, k}\phi(B_{\xi})^k$$
and $\underbar{g}_{\xi, k} \in \A$ for $k\leq m_\xi$ and  $\overline{\underbar{h}_\xi \left[{D_i} \right]} \subseteq U_1^\xi$ and  
$\overline{\underbar{h}_\xi \left[{E_i} \right]} \subseteq U_2^\xi.$

As there are countably many such triples $U_1^\xi, U_2^\xi, m_\xi$,  this means that there are 
distinct $\xi^j<\cc$ for $j<4$ and there are $g_{j, k}\in \A$ for $k\leq m$ and $j<4$ for some $m\in \N$ 
and open $U_1, U_2\subseteq\R$ (which are finite unions of bounded open intervals with rational end-points) distant
by some $\delta>0$ such that for  
each $j<4$ and
$$h_j= g_{j, 0} + \sum_{1 \leq k\leq m}g_{j, k}\phi(B_{\xi^j})^k\leqno (1)$$
and we have 
$$\overline{h_j \left[{D_i} \right]} \subseteq U_1 \ \ \hbox{and} 
\ \ \overline{h_j \left[{E_i} \right]} \subseteq U_2.\leqno (2)$$

Let $\V$ be a finite open cover of $K$ consisting of $\A$-maximal sets (Definition \ref{def-maximal}) such that for
 each $V \in \V$ and each $j < 4$ we have that $\operatorname{osc} g_{j, 0}(V) \leq \delta/3$. 
 To see that such a cover exists note that for every $r \in \R$ and $j < 4$ the preimage
  $(g_{j, 0})^{-1} [(r,r+\delta/3)]$ is an open $\A$-maximal set by 
  Lemma \ref{lemma-A-maximal} (6) as $g_{j, 0} \in \A$ 
  and the value $\operatorname{osc}g_{j, 0}$ is clearly at most $\delta/3$ on such sets.

Fix $V \in \V$. For $j < 4$ define
\begin{align*}
    X_j = \coz(\phi(B_{\xi^j})).
\end{align*}
Note that by (1) for each $j < 4$ we have 
$$h_j | (K \setminus X_j) = g_{j, 0} |(K \setminus X_j).\leqno (3)$$
Now for distinct $j, j'<4$ define
$$
    \hat{h}_{j, j'} := g_{j, 0} + \sum_{1 \leq k \leq m} g_{j, k}  
    \left(\sum_{n \in (B_{\xi^j} \cap B_{\xi^{j'}})} f_n \right)^k.\leqno (4)
$$
Then clearly $\hat{h}_{j, j'} \in \A$  as  $B_{\xi^j} \cap B_{\xi^{j'}}$ is finite for distinct $j, j'<4$
and $g_{j, k}, f_n\in \A$ for $n\in \N$ and $k\leq m$. Let us calculate the values
of the functions $\hat{h}_{j, j'}$ inside and outside of the sets $X_j\cap X_{j'}$ for distinct $j, j'<4$.

First note that by Lemma \ref{morphism-nice} for distinct $j , j'< 4$ we have 
$$\coz\Bigg(\sum_{n \in (B_{\xi^j} \cap B_{\xi^{j'}})} f_n\Bigg)\subseteq X_j\cap X_{j'},$$
so
$$\hat{h}_{j, j'} | (K \setminus (X_j\cap X_{j'})) = g_{j, 0} |(K \setminus (X_j\cap X_{j'})).\leqno (5)$$
Secondly, it follows from Lemma \ref{morphism-nice} that
\begin{align*}
    \phi(B_{\xi^j})|(X_j\cap X_{j'}) =
     \Bigg(\sum_{n \in (B_{\xi^j} \cap B_{\xi^{j'}})} f_n\Bigg)|(X_j\cap X_{j'}),
\end{align*}
and so by (1) and (4) we obtain
$$\hat{h}_{j, j'} | (X_j\cap X_{j'}) = {h}_j | (X_j\cap X_{j'}), \leqno(6)$$
for distinct $j, j'<4$.

Now note that
for each $j < 4$ one of the following three cases holds:
\begin{enumerate}[a)]
    \item $(D_i \cap V) \subseteq X_j$ and $(E_i \cap V) \subseteq X_j$;
    \item $(D_i \cap V) \subseteq X_j$ and $(E_i \cap V) \not\subseteq X_j$;
    \item $(E_i \cap V) \subseteq X_j$ and $(D_i \cap V) \not\subseteq X_j$.
\end{enumerate}
Indeed, if none of the above options hold, we can pick $x_0 \in (D_i \cap V) \setminus X_j$
 and $x_1 \in (E_i \cap V) \setminus X_j$. But then by (3) we have
\begin{align*}
    |h_j(x_0) - h_j(x_1)| = |g_{j, 0}(x_0) - g_{j, 0}(x_1)| \leq \operatorname{osc} g_{j, 0} (V) \leq \delta/3
\end{align*}
which contradicts the fact that $\operatorname{dist}(U_1,U_2) = \delta$ and (2). 
Since we have three cases  and four numbers $j < 4$, 
there are $j_1,j_2 < 4$ which satisfy the same case. Let us, for simplicity, assume that $j_1 = 0$ and $j_2 = 1$.

We are now going to show that in each of the cases a) - c) both of
the functions $\hat{h}_{0, 1}$ and $\hat{h}_{1, 0}$ separate $D_i \cap V$ and $E_i \cap V$.

First, suppose that the case a) holds for $j=0,1$. Then for distinct $j, j'\in\{0,1\}$  by (6) and (2)  we have
$$ \overline{\hat{h}_{j, j'} [D_i \cap V]} = 
    \overline{h_j [D_i \cap V]} \subseteq U_1, $$
    $$\overline{\hat{h}_{j, j'} [E_i \cap V]} = 
    \overline{h_j [E_i \cap V]} \subseteq U_2.$$
and so both $\hat{h}_{0, 1}$ and $\hat{h}_{1, 0}$ separate $D_i \cap V$ and $E_i \cap V$.

Suppose that the case b) holds for $j=0,1$. 
Then as in the previous case for distinct $j, j'\in\{0,1\}$
by (6) and (2) we have
$$ \overline{\hat{h}_{j, j'} [D_i \cap V]} = 
    \overline{h_{j} [D_i \cap V]} \subseteq U_1.\leqno (7)$$
    Now let $x_j\in (E_i\cap V)\setminus X_j$. 
As $g_0^j(x_j)=h_j(x_j)\in U_2$  by (2) and (3), it follows from the fact that $V\in \mathcal V$ that 
\begin{enumerate}[(8)]
\item $g_{j, 0}[V]\subseteq \{x\in \R: \operatorname{dist}(x, U_2)\leq\delta/3\}.$
\end{enumerate}
So, by (6), (5) (8) and (2)
$$\hat{h}_{j, j'}[E_i\cap V]\subseteq 
\hat{h}_{j, j'}[X_0\cap X_1\cap E_i\cap V]\cup \hat{h}_{j, j'}[V\setminus(X_0\cap X_1)]
\subseteq \{x\in \R: \operatorname{dist}(x, U_2)<\delta/3\},$$
and  so combining the above with (7) we conclude that
  $\overline{\hat{h}_{j, j'} [E_i \cap V]}$ and  $\overline{\hat{h}_{j, j'} [D_i \cap V]}$
 are disjoint by the choice of $\delta$ as needed for the separation of $D_i\cap V$ and $E_i\cap V$ by $\hat{h}_{1,0}$
 or by $\hat{h}_{0,1}$.
The case c) is analogous to the case b).

Hence, we have found for every $V \in \V$ a function from $\A$ which separates $D_i \cap V$ and $E_i \cap V$, 
and thus by Lemma \ref{piecewise-separation}, we get that $\A$ separates $D_i$ and $E_i$ -- a contradiction
with the hypothesis of the lemma.
\end{proof}

\subsection{Algebras with few$^*$ operators}

Recall from the Introduction when a Banach space of the form $C(K)$
has few$^*$ operators and what is a weak multiplier.

\begin{proposition}\label{few-star}
Suppose that $L$ is a compact Hausdorff almost $P$-space such that
for every pairwise disjoint $\{f_n :n\in \N\} \subseteq C_1(L) \setminus \{0\}$ and every pairwise disjoint 
$\{V_m: m\in \N\} \subseteq \coz(L)$ such that $f_n|V_m=0$ for every $n, m\in \N$
there is a sublattice $\mathcal P$ 
of $\wp(\N)$ with the least element $\emptyset$ and a submorphism $\phi: \mathcal P \rightarrow C_1(L)$ such that 
\begin{enumerate}[(a)]
\item $\mathcal P$ contains all finite sets and for every
 $A \in \mathcal{P}$ and $k \in \N$, the set $A \setminus k \in \mathcal{P}$,
\item $\phi(\{n\})=f_n$ for each $n\in \N$,
\item $\phi(A)|V_m=0$ for every $m\in \N$ and $A \in \mathcal{P}$,
\item for every infinite $A\subseteq\N$  there is $B\subseteq A$ such that
 \begin{enumerate}[(i)]
 \item $B\in \mathcal P$
 \item $C_1(L)$ does not separate $\bigcup\{V_m: m\in B\}$ and $\bigcup\{V_m: m\in A\setminus B\}$
 \end{enumerate}
\end{enumerate}
Then $C(L)$ has few$^*$ operators.
\end{proposition}

\begin{proof}
Suppose  that $T:C(L)
\rightarrow C(L)$ is not a weak multiplier. By 2.1 of \cite{few} 
there is an $\varepsilon>0$,  and a pairwise  disjoint
sequence $\{f_n: n\in \N\}\subseteq C_1(L)$  and points $(x_n)_{n \in \N} \subseteq L$ such that for every $n \in \N$ we have
$f_n(x_n)=0$ and 
$$|T(f_n)(x_n)|>\varepsilon. \leqno 1)$$
Since $T$ is bounded we may conclude that $\{x_n:n\in \N\}$ is infinite and so
by passing to a an infinite subset we may assume that there are open pairwise disjoint $V_n'$s in $L$
such that $x_n\in V_n'$ for every $n\in \N$.
By applying the Ramsey theorem \cite[Theorem 9.1]{jech}
to the coloring $c: [\N]^2\rightarrow \{0,1\}$ where $c(n,m)=1$ if and only if $f_n(x_m)\not=0$ or $f_m(x_n)\not=0$,
and the fact that $\{f_n: n\in \N\}$ is pairwise disjoint, we may assume, by passing to an infinite subset of $\N$, that $f_n(x_m)=0$
for any two distinct $n,m\in \N$.

By applying the Rosenthal lemma \cite[Lemma 1.1]{rosenthal} and passing to a
 further infinite subset of $\N$, we may assume that
$\sum_{n\in \N\setminus\{m\}} |T(f_n)(x_m)|<\varepsilon/3$ holds for all $m\in \N$.
For each $m\in\N$ the set
$$\bigcap_{n\in\N}f_n^{-1}[\{0\}]\cap \bigcap_{n\in\N\setminus{m}}(T(f_n))^{-1}[\{T(f_n)(x_m)\}]
\cap(|T(f_m)|)^{-1}[(\varepsilon, \infty)]\cap V_m'$$
is a $G_\delta$ subset of $L$ which is not empty as witnessed by $x_m$. As $L$ is an almost $P$-space
this set has a nonempty interior $V_m$.
So $\{V_m: m\in \N\}\cup\{\coz(f_n): n\in \N\}$ is pairwise disjoint and
$$\sum_{n\in \N\setminus\{m\}} ||T(f_n)|V_m||_\infty < \varepsilon/3\leqno 2)$$
for all $m\in \N$ as well as for all $x\in V_m$ we have 
$$|T(f_m)(x)|>\varepsilon.\leqno 3)$$
Note that it follows from 2) that for any $A \subseteq \N$ we have that $\sum_{n \in A}T(f_n)|V_m$
 is a well defined bounded continuous function on $V_m$ for any $m \in \N$.

Let $\phi$ and $\mathcal P$ be as in the statement  of the proposition
for the pair $\{f_n: n\in \N\}$ and $\{V_m: m\in \N\}$.

\noindent{\bf Case 1.} There is an infinite $A\subseteq\N$
 and there are nonempty sets $(U_m)_{m \in \N} \subseteq \coz(L)$ with $U_m \subseteq V_m$ for each $m \in \N$, such that  
$$T(\phi(B))|U_m=\sum_{n\in B} T(f_n)|U_m$$
holds for every $B\subseteq A$, $B\in\mathcal{P}$.
\vskip 6pt
\noindent In this case, by 2) and 3), if $B\subseteq A$ and $m\in B$, then we have for every $x\in U_m$ that
$$|T(\phi(B))(x)| = \left| \sum_{n\in B} T(f_n)(x) \right| > 2 \varepsilon/3$$
and if $m\not \in B$, then for every $x\in U_m$ we have
$$|T(\phi(B))(x)| = \left| \sum_{n\in B} T(f_n)(x) \right| < \varepsilon/3$$
Then $T(\phi(B))$ separates $\bigcup\{U_m: m\in B\}$ from $\bigcup\{U_m: m\not\in B\}$ 
whenever $B \subseteq A$, $B\in\mathcal{P}$. This contradicts  
condition (d) of the proposition.
\vskip 6pt
\noindent{\bf Case 2.} Case 1 does not hold.
\vskip 6pt
\noindent Assuming the negation of the condition from Case 1, we will carry out
a transfinite recursive construction which will contradict the boundedness of the operator $T$. 
Let $\{A_\xi: \xi<\omega_1\}$ be an almost disjoint (i.e., such that pairwise intersections
of its elements are finite) family 
of infinite subsets of $\N$. 
For all $\xi<\omega_1$ construct:
\begin{itemize}
\item an infinite $B_\xi\subseteq A_\xi$ such that $B_\xi \in \mathcal{P}$,
\item a pairwise disjoint $\{U^\xi_m: m\in \N\}\subseteq \coz(L)$ such that 
$\emptyset \neq \overline{U_m^{\xi'}}\subseteq U_m^{\xi}\subseteq V_m$ for all $m\in \N$
and all $\xi < \xi'<\omega_1$,
\item $m_\xi, n_\xi\in \N\setminus\{0\}$,
\end{itemize}
such that for all $x\in U^\xi_{m_\xi}$ we have
$$\left| T(\phi(B_\xi))(x)-\sum_{n\in B_\xi} T(f_n)(x) \right| > 1/n_\xi.\leqno 4)$$

We make this construction recursively. Suppose we have already carried the construction out for all $\xi < \xi' < \omega_1$.
 We find for each $m \in \N$ open sets $V^{\xi'}_m$
  such that $V^{\xi'}_m \subseteq \bigcap_{\xi < \xi'} U^{\xi}_m$: if $\xi' = 0$, 
  we take $V^{\xi}_m = V_m$, if $\xi' = \xi + 1$ is a successor ordinal, we can take $V^{\xi'}_m = U^{\xi}_m$,
   and finally if $\xi'$ is a limit ordinal, we use the fact that $L$ is an almost $P$-space
    to find such $V^{\xi'}_m$ (as the set $\bigcap_{\xi < \xi'} U^{\xi}_m = \bigcap_{\xi < \xi'} \overline{U^{\xi}_m}$ 
    is nonempty $G_\delta$, and so has nonempty interior). 
    Now, as Case 1 does not hold, it is false in particular for $A = A_{\xi'}$ and $U_m = V^{\xi'}_m$. 
    Hence, we can find $B_{\xi'} \in \mathcal{P}$ and $m_{\xi'},n_{\xi'} \in \N$ 
    such that for some $x \in V^{\xi'}_{m_{\xi'}}$ we have
$$\left| T(\phi(B_{\xi'}))(x) - \sum_{n \in B_{\xi'}} T(f_n)(x) \right| > 1/n_{\xi'}.$$
Now we pick a cozero neighbourhood $U_{m_{\xi'}}^{\xi'}$ of $x$ such that the above
 inequality holds on $U_{m_{\xi'}}^{\xi'}$ and $\overline{U_{m_{\xi'}}^{\xi'}} \subseteq V^{\xi'}_{m_{\xi'}}$
  (for this we use that $\sum_{n \in B_{\xi'}} T(f_n)$ is continuous on $V_m$'s, which follows from 2)). 
  For any other $m \neq m_{\xi'}$ we pick any nonempty cozero $U^{\xi'}_m$ such
   that $\overline{U^{\xi'}_m} \subseteq V^{\xi'}_m$.

A single pair $(m', n')$ has appeared infinitely many times as $(m_\xi, n_\xi)$. Let $i\in \N$ be 
such that $i/3n'>||T||$ and consider $\xi_1<...< \xi_i$ such that $(m_{\xi_i}, n_{\xi_i})=(m', n')$.
Let $k\in \N$ be such that the pairwise intersections of all $B_{\xi_1}, ..., B_{\xi_i}$
are included in $\{0, ..., k-1\}=k$ and such that 
$$\sum_{n\in \N\setminus k} ||T(f_n)|V_{m'}||_\infty < 1/3n'.\leqno 5)$$
The latter condition follows from the convergence of the series in 2).

We set $B_\xi'=B_\xi\setminus k$ for $\xi < \omega_1$ and note that $B_\xi' \in \mathcal{P}$ by
hypothesis (a) of the proposition. Then 
$$\phi(B_\xi) = \phi(B_\xi' \vee (B_\xi \cap k)) = \phi(B_\xi') \vee \phi(B_\xi \cap k)
= \phi(B_\xi')+\sum_{n\in B_\xi\cap k} f_n.$$
We conclude from 4) that 
$$\left| T(\phi(B_{\xi_j}'))(x)-\sum_{n\in B_{\xi_j}'} T(f_n)(x) \right| > 1/n'$$
for all $x\in U^{\xi_j}_{m'}$ and all $1\leq j\leq i$. 
By 5) this implies that
$|T(\phi(B_{\xi_j}'))(x)|>2/3n'$ for
all $x\in U^{\xi_j}_{m'}$ for all  $1\leq j\leq i$, and consequently for some  $F\subseteq \{1, ..., i\}$ 
of cardinality not smaller than $i/2$ and any $x\in \bigcap_{j\in F}U^{\xi_j}_{m'}$ we have
$$\left| \sum_{j\in F}T(\phi(B_{\xi_j}'))(x) \right| > (2/3n')(i/2)>||T||.$$
However, as $B_{\xi_j}'$, $1 \leq j \leq i$ are pairwise disjoint, so are $\phi(B_{\xi_j}')$, $1 \leq j \leq i$.
As $\phi(B_{\xi_j}')$'s are in $C_1(L)$, it  follows that $||\sum_{j\in F}\phi(B_\xi')|| \leq 1$ and so
$||T(\sum_{j\in F}\phi(B_{\xi_j}'))||\leq ||T||$, a contradiction.
\end{proof}

\section{Connectedness of Gelfand spaces}

In  this section we fix a Banach space $\XX$ of density
 at most continuum and set $\X= B_{\XX^*} \times \R^2$, where $B_{\XX^*}$ is
 the unit dual ball of $\XX$ with the weak$^*$ topology. 
 Then $\X$ is a locally compact $\sigma$-compact noncompact 
 connected Hausdorff space of weight at most $\cc$ and so by Propositions \ref{weight-remainder}, \ref{X*-Fspace},
  and \ref{remainder-P-space}
 the \v Cech-Stone remainder $\X^*$ of $\X$ is an $F$-space and an almost $P$-space of weight at most $\cc$, so
 the results of Section 4 apply to $\X^*$. The aim of this section is to
 prove a sufficient condition for  an algebra $\A\subseteq C(\X^*)$
 which implies that $\nabla\A\setminus F$ is connected for every finite $F\subseteq\nabla\A$ 
 (Proposition \ref{main-connected}). All arguments in this section are simple generalizations of
 geometric arguments for $(\R^2)^*$ using the fact that $B_{\XX^*}$ is locally convex and compact  in the
 weak$^*$ topology. Some other work with the connectedness in
 the \v Cech-Stone remainders can be found in \cite{kp, alan-kp, zhu}.
 
 A general element of $\X$ will be denoted as $(x,y,z)$
  where $x \in B_{\XX^*}$ and $y,z \in \R$. If $(y,z) \in \R^2$, we use $|(y,z)|$
   to denote its Euclidean norm $\sqrt{y^2 + z^2}$.
Before proceeding, let us recall the following well known results for further reference
 (see e.g. Theorem 6.1.9. and Corollaries 6.1.11, 6.1.13, 6.1.19. of \cite{engelking}).

\begin{lemma} \label{connected-union} Let $Y$ be a topological Hausdorff space.
\begin{enumerate}
    \item If  $Z_i$, for ${i \in I}$, and $Z$ are 
     connected subspaces of $Y$ such that $Z \cap Z_i \neq \emptyset$ for every $i \in I$,
     then $Z \cup \bigcup_{i \in \I} Z_i$ is connected.
     \item If $Z\subseteq Y$ is dense in $Y$ and connected, then $Y$ is connected.
     \item If $(Z_n)_{n \in \N}$ is a  decreasing sequence of connected compact subspaces of $Y$,
     then $\bigcap_{n\in \N}Z_n$ is connected and compact.
    \item If any two points of $Y$ belong to a connected subspace of $Y$, then $Y$ is connected.
\end{enumerate}
\end{lemma}

\begin{definition} \label{definition-functions}
We use the following notation.
\begin{enumerate}
\item  $B(v,r) = \{u \in \R^2 : |u-v| < r\}$
for $v \in \R^2$ and $r\in\R_+$.
\item $\overline{B}(v,r) = \{u \in \R^2 : |u-v| \leq r\}$
for $v \in \R^2$ and $r\in\R_+$.
\item $l(\alpha, I) = \{r \alpha: r \in I\}$ for $\alpha$ in the unit sphere of $\R^2$ and $I\subseteq \R_+$.
\item $U(Y, \epsilon) = \{(y,z) \in \R^2: \operatorname{dist}((y,z),Y) < \epsilon\}$ for $Y\subseteq \R^2$
and $\epsilon>0$.
\item $S_{I}=\{(x,y,z) \in \X: |(y,z)| \in I\}$ for $I\subseteq\R_+$. 

\end{enumerate}
\end{definition}

The following lemma contains several statements about the above sets.
Recall that by an \textit{interval} we mean a connected subset of $\R$.

\begin{lemma} \label{remark-level-sets}
    The following holds true:
    \begin{enumerate}
        \item Let $I \subseteq \R_+\cup\{0\}$ be an interval. Then $S_I$ is connected. 
        Moreover, if $I$ is open (resp. closed, resp. has nonempty interior) in $\R_+\cup\{0\}$, then 
        $S_I$ is open (resp. closed, resp. has nonempty interior) in $\X$.
      \item Let $I, J \subseteq \R_+$ be  intervals and   $\alpha$ be in the unit sphere of $\R^2$.
      Then  $I\cap J\not=\emptyset$ if and only if 
      $S_I\cap (B_{\XX^*}\times l(\alpha, J))\not=\emptyset$.
        \item For every interval $J\subseteq \R_+$, $\varepsilon>0$ and $\alpha$ in the unit sphere of $\R^2$ the
        sets $l(\alpha, J)$, $B_{\XX^*}\times l(\alpha, J)$,
         $U(l(\alpha, J), \epsilon)$ and $B_{\XX^*}\times U(l(\alpha, J), \epsilon)$ are connected.
    \end{enumerate}
\end{lemma}

\begin{proof}
    The point (1) follows from the fact that $S_I = B_{\XX^*} \times \{(y,z) \in \R^2: |(y,z)| \in I\}$
     and this set is connected and open (resp. closed, resp. has nonempty interior) if and only if $I$ is.
     Similarly (2) follows from the fact that $|\alpha x|=x$ if $x\in \R_+$ and $\alpha$
     is in the unit sphere of $\R^2$.

    For (3) note that $l(\alpha, J)$ is homeomorphic to the interval $J$ and so connected.
    As 
    $$U(l(\alpha, J), \epsilon)=\bigcup\{B(v,\epsilon): v\in l(\alpha, J)\}$$
    and $l(\alpha, J)$ is connected and $v\in B(v,\epsilon)$, so Lemma \ref{connected-union} (1) applies
    and yields the connectedness of $U(l(\alpha, J), \epsilon)$. The remaining parts
    of (3) follow from the fact that products of connected spaces are connected.
   
\end{proof}

\begin{lemma} \label{lemma-connected-separation2}
Suppose that $F, G\subseteq \X^*$ are nonempty closed and disjoint and that
for every Cantor system $(f_s)_{s\in  s\in \{0,1\}^{<\omega}} \subseteq C_1(\X)$
there is $\sigma\in \{0, 1\}^\N$ such that $(\beta f_\sigma)|F=0$. Then there are 
  \begin{itemize}
  \item $(m_n)_{n\in \N}\subseteq\N$, 
  \item $(a_n)_{n\in \N}$, $(b_n)_{n\in \N}$ with $0<a_0\leq a_n<b_n<a_{n+1}$ for
  all $n\in \N$ and $\lim_{n\rightarrow \infty}a_n = \infty$,
  \item  $(\alpha_{n, m})_{m\leq m_n, n\in\N}$ in the unit sphere of $\R^2$,
  \item $0<\varepsilon_n<(b_n-a_n)$ for all $n\in \N$,
 \item open sets
$U_k', U_k'',  U_{n, m}'', U_k''', U_{n, m}'''\subseteq \X$ for $k, n\in\N$ and $m\leq m_n$ 
\end{itemize}
such that
\begin{enumerate}[(a)]
\item $U_k'=\bigcup_{n\geq k}S_{(a_{n}, b_{n})}$,
\item $U_{n, m}''=C_{n, m}\times B_{n, m}\subseteq S_{(a_n, b_{n+1})}$,
where $C_{n, m}$ is a nonempty open convex subset of $B_{\XX^*}$
and $B_{n, m}\subseteq\R^2$ is a open ball,
\item $U_k''=\bigcup_{n\geq k}\bigcup_{m\leq m_{n}} U_{n, m}''$,
\item $U_{n, m}'''=B_{\XX^*}\times U(l(\alpha_{n, m}, [b_n, a_{n+1}]), \epsilon_n)$,
\item $U_k'''=\bigcup_{n\geq k}\bigcup_{m\leq m_{n}} U_{n, m}'''$,
\item $B_{n, m}\cap l(\alpha_{n, m}, [b_n, a_{n+1}])\not=\emptyset$ for each $n\in \N$ and $m\leq m_n$,
\item $G\subseteq \overline{U_k'}\cup \overline{U_k''}$ for each $k\in \N$, where the closure is taken in $\beta X$
\item $F\cap (\overline{U_k'}\cup\overline{U_k''}\cup\overline{U_k'''})=\emptyset$ for each $k\in \N$, 
where the closures are taken in $\beta X$
\end{enumerate}
\end{lemma}

\begin{proof} 

The main idea is that we choose the sets ${U_k''}$'s so that (g) holds, while
the sets $U_k'$'s and $U_k'''$'s take care of the the connectedness
of each union $\overline{U_k'}\cup\overline{U_k''}\cup\overline{U_k'''}$ in the proof of Lemma \ref{big-connected}.
Perhaps to understand the role of these sets fully the reader could look at
the rest of this section before studying the details of this proof.

Use the normality of $\beta\X$ to find open sets $U, W\subseteq \beta \X$ such that
\begin{enumerate}
\item $G\subseteq U\subseteq \overline{U}\subseteq W$,
\item $\overline{W}\cap F=\emptyset$,
\end{enumerate}
 where the closure is taken in $\beta\X$.
 Note that $G\subseteq U\setminus S_{[0,a]}$ for every $a>0$ and
 so such sets $U\setminus S_{[0,a]}$ are open and not empty and hence
 intersect with $X$ by the density of $X$ in $\beta X$. It follows that there is
 $A\subseteq \R_+$ which is unbounded in $\R_+$ such that 
 \begin{enumerate}[(3)]
\item $U\cap S_{\{a\}}\not=\emptyset$ for all $a\in A$.
\end{enumerate}

\noindent{\bf Step 1. Constructing the sets $U_k'$ for $k\in \N$.} 
First we construct $(a_n)_{n\in \N}$,  $(b_n)_{n\in \N}$ and we put
$U_k'=\bigcup_{n\geq k}S_{(a_{n}, b_{n})}$ as in (a)
 and we will prove the part of (h) that says that $F\cap \overline{\bigcup_{n\geq k}S_{(a_{n}, b_{n})}}=\emptyset$ 
 for $k\in \N$. 
 
 For $n>1$ consider $g_n\in C_1(\R_+)$ such that 
 $$g_n^{-1}[\{1\}]=[n, n+1], \ 
 g_n|(\R_+\setminus[n-1, n+2])=0$$
 for every $n\in\N$. So $\{g_{4n}: n>0\}$ have  pairwise disjoint supports.
 Relabel  $\{g_{4n}: n>0\}$ as $\{g_s: s\in \{0,1\}^{<\omega}\}$ and let $I_s$ be
 the closed interval such that $I_s=g_s^{-1}[\{1\}]$.  Define $f_s\in C_1(X)$ by
 $$f_s(x, y, z)= g_s(|(y, z)|)$$
 Note that $f_s^{-1}[\{1\}]=S_{I_s}$. Moreover 
 $(f_s)_{s\in \{0,1\}^{<\omega}}$ is a Cantor system as in Definition \ref{cantor-system}.
 Items (1) and (2) of Definition \ref{cantor-system} are clear, to see (3) note that
 the sets $S_{[0,a)}$ for $a>0$ form an open cover of $\X$ and
 they intersect only finitely many supports of $f_s$'s.
 
 Hence, $(f_s)_{s\in \{0,1\}^{<\omega}}$ is a Cantor system.
 So, by the hypothesis on $F$ we obtain $\sigma\in \{0,1\}^\N$ such that
 $(\beta f_\sigma)|F=0$, where $f_\sigma$ is as in Definition \ref{cantor-system}. 
  In particular, $(\beta f_\sigma)^{-1}[\{1\}]$ is a closed set
 disjoint from $F$ and contains all sets $f_{\sigma|n}^{-1}[\{1\}]$ for  all $n\in \N$
 as $(\beta f_\sigma)|X=\sum_{n\in \N} f_{\sigma|n}$. As $f_{s}^{-1}[\{1\}]=S_{I_s}$, 
 we obtain $F\cap \overline{\bigcup_{n\in \N}S_{I_{\sigma|n}}}=\emptyset$.

Labeling the endpoints of $I_{\sigma|n}$'s by appropriate $a_n$ and $b_n$ so that
 $0<a_0\leq a_n<b_n<a_{n+1}$ for
  all $n\in \N$ and $\lim_{n\rightarrow \infty}a_n = \infty$ we obtain 
  $$F\cap \overline{\bigcup_{n\in \N}S_{(a_n, b_n)}}=\emptyset,$$
 where the closure is taken in $\beta\X$. Thus
 we have  (a)
 and  the part of (h) that says that $F\cap \overline{U_k'}=\emptyset$ 
 for $k\in \N$. 
 Moreover we may assume that 
 \begin{enumerate}[(4)]
 \item $(b_n, a_{n+1})\cap A\not=\emptyset$ for every $n\in \N$,
 \end{enumerate}
where $A\subseteq \R_+$ is as in (3). This can be arranged by thinning out the sequence
of $(a_n, b_n)_{n\in \N}$.

\vskip 6pt
\noindent{\bf Step 2. Constructing the sets $U_k''$ for $k\in\N$.} 
In this step we will construct  $(m_n)_{n\in\N}$
and  open convex subsets $C_{n, m}$ of $B_{\XX^*}$
and  balls $B_{n, m}\subseteq\R^2$ for $m\leq m_n$ and 
$n\in \N$ and hence $U_{n, m}$'s and $U_k''$'s (as defined in (b) and (c)) for $n, k\in\N$
$m\leq m_n$
 and we will prove (g) and the part of (h) that says that $F\cap \overline{U_k''} =\emptyset$
 for $k\in\N$. 
 
Let $U\subseteq \beta\X$ be as in (1)-(2). Fix $n\in\N$.
 For each element $u$ of $\overline{U \cap S_{[b_n, a_{n+1}]}} \subseteq W \cap S_{(a_n,b_{n+1})}$ 
  we can pick a basic neighborhood of the form $C_u \times B_u$, 
  where $C_u \subseteq B_{\XX^*}$ is open and convex and $B_u$
   is an open ball in $\R^2$ (recall that the weak$^*$ topology is locally convex), 
   such that 
   $$\overline{C_u \times B_u} \subseteq W \cap S_{(a_n, b_{n+1})}.\leqno (5)$$
   As $\overline{U \cap S_{[b_n, a_{n+1}]}}$ is compact, we can take a finite subcover of
    its cover $\{C_u \times B_u: u \in \overline{U \cap S_{[b_n, a_{n+1}]}}\}$ and
     enumerate it as $\{U_{n, 0}'',\dots, U_{n, m_n}''\}$. So, (b) of the lemma is satisfied.
   Here (3)-(4) is used so that we have  $U \cap S_{[b_n, a_{n+1}]}\not=\emptyset$  which guarantees
   that the subcover is not empty, i.e., $m_n\geq 0$ (this is used in the proof of Lemma \ref{big-connected}).
     
     By Lemma \ref{closure-remainder} for every $k\in\N$ we have
     $$G\subseteq \overline {U\cap X}=\overline{U\cap S_{[0, a_k]}}\cup \overline{U\cap S_{(a_k, \infty)}}.$$
     Since $G\subseteq  \X^*$ and $\overline{ S_{[0, a_k]}}\cap X^*=\emptyset$ as  $S_{[0, a_k]}\subseteq X$
     is compact, we conclude that  for every $k\in\N$ we have
     $$G\subseteq \overline{U\cap S_{(a_k, \infty)}}= \overline{U_k'}\cup 
     \overline {\bigcup_{n\geq k}(U\cap S_{[b_n, a_{n+1}]})}.$$
    As $\{U_{n, 0}'',\dots, U_{n, m_n}''\}$ is a cover 
     of $\overline{U \cap S_{[b_n, a_{n+1}]}}$ we conclude (g).
     
     Also note that $\overline{\bigcup_{n \in \N} U_n''} \subseteq \overline{W}$ by (5),  so
     (1) implies the part of (h) which says that $F\cap \overline{U_k''} =\emptyset$.
    
\vskip 6pt
\noindent{\bf Step 3. Constructing the sets $U_k'''$ for $k\in\N$.} 
In this step for $m\leq m_n$ and $n\in\N$ we will construct  elements $\alpha_{n, m}$ of the 
unit sphere of $\R^2$  so that (f) is satisfied
and the part of (h) which says that $F\cap \overline{U_k'''} =\emptyset$ holds for
every $k\in\N$, where $U_k'''$ are defined in (e) using (d).

Note that by the construction in Step 2  for each
$n\in \N$ and  $m\leq m_n$ the set  $B_{n, m}$ is an open ball with the center in
$u\in \{(y,z)\in \R^2: |(y,z)|\in [b_n, a_{n+1}]\}$  so
  it must be true that $B_{n, m} \cap \{(y,z) \in \R^2: |(y,z)| \in (b_n, a_{n+1})\}$
   is nonempty and open. So, for each $n\in \N$ we can pick $(y_{n, m, s}, z_{n, m, s})$
    in the above intersection for all $m\leq m_n$ and
   $s\in \{0,1\}^{n}$, such that their arguments $|(y_{n, m, s}, z_{n, m, s  })|^{-1}
    (y_{n, m, s}, z_{n, m, s})$ are all distinct (as a nonempty open subset of $\R^2$
     has uncountably many elements with distinct arguments). Put $\alpha_{n, m, s}
      = |(y_{n, m, s}, z_{n, m, s  })|^{-1} (y_{n, m, s}, z_{n, m, s})$.
   It follows that 
   $$l(\alpha_{n, m, s}, [b_n, a_{n+1}])\cap B_{n, m}\not=\emptyset$$
   for every  $m\leq m_n$ and
   $s\in \{0,1\}^{n}$ and $n\in \N$. 
   
  Now for each $n\in \N$ pick $0<\epsilon_n' < \frac{1}{2} \min \{b_n-a_n, b_{n+1}-a_{n+1}\}$ 
small enough so that the compact disjoint sets $l(\alpha_{n, m, s}, [b_n, a_{n+1}])$ for 
$m\leq m_n, s \in \{0,1\}^n$, are $2\epsilon_n'$-separated and with $\epsilon_n'$ converging  to $0$. 
For $s\in \{0,1\}^{<\omega}$ define  $f_s\in C_1(X)$  by
$$f_s((x,y,z)) = 1 - \min \left(1,(\epsilon_n')^{-1} \operatorname{dist}((y,z),\bigcup_{m=0}^{m_n}
 l(\alpha_{n, m, s}, [b_n, a_{n+1}])) \right).$$

We note that
\begin{enumerate}
    \item[(6)] $\supp(f_s)= B_{\XX^*} \times \bigcup_{m=0}^{m_n}
     \overline{U(l(\alpha_{n, m, s}, [b_n, a_{n+1}]),\epsilon_n')}\subseteq
      S_{[b_n-\epsilon_n', a_{n+1}+\epsilon_n']}$,
    \item[(7)] $f_s^{-1}[[1/2, 1]]\supseteq B_{\XX^*} \times \bigcup_{m=0}^{m_n}
     \overline{U(l(\alpha_{n, m, s}, [b_n, a_{n+1}]),\epsilon_n'/2)}$.
   
    \end{enumerate}
  So, supports of $f_s$'s are compact. They are also pairwise disjoint, which  follows from the choice of
  $\epsilon_n'$ and (6). Moreover, only finitely many $f_s$'s have supports intersecting
   $S_{[0, b_n-\epsilon_n')}$. However, such sets form an open cover of $X$ because $b_n \rightarrow \infty$ and
   $\epsilon_n' \rightarrow 0$, so
   each point of $X$ has an open neighborhood which intersects the supports of only
   finitely $f_s$'s. 
   
   Consequently, $(f_s)_{s\in \{0,1\}^{<\omega}}$ is a Cantor system as in
   Definition \ref{cantor-system}. It follows from the hypothesis of the
   lemma that there is $\sigma\in\{0,1\}^\N$ such that 
   $f_\sigma|F=0$.  
    In particular $(\beta f_\sigma)^{-1}[[1/2, 1]]\subseteq\beta\X$ is a closed set
 disjoint from $F$ and containing all sets $f_{\sigma|n}^{-1}[[1/2, 1]]$ for  all $n\in \N$
 as $(\beta f_\sigma)|X=\sum_{n\in \N} f_{\sigma|n}$. As 
 $f_s^{-1}[[1/2, 1]]$ satisfies (7) above 
 we obtain 
 that
  $$F\cap \overline{\bigcup_{n\in \N}\bigcup_{m\leq m_n}
  B_{\XX^*} \times U(l(\alpha_{n, m, \sigma|n}, [b_n, a_{n+1}]),\epsilon_n'/2)}=\emptyset.$$
 
Putting $\alpha_{n, m}=\alpha_{n, m, {\sigma|n}}$ 
and $\epsilon_n=\epsilon_n'/2$  we obtain 
  $F\cap \overline{U_k'''}=\emptyset$ for each $k\in \N$, where
 the closure is taken in $\beta\X$.   
\end{proof}

\begin{lemma} \label{big-connected}
Suppose that $F\subseteq \X^*$ is closed and such that
for every Cantor system $(f_s)_{s\in \{0,1\}^{<\omega}}\subseteq C_1(\X)$
there is $\sigma\in \{0, 1\}^\N$ such that $(\beta f_\sigma)|F=0$. 
Then $X^*\setminus F$ is connected. 
\end{lemma}
\begin{proof}
Let $y_0, y_1$ be any two distinct points in $\X^* \setminus F$. 
By Lemma \ref{connected-union} (4) it is enough to find a connected
subspace $H$ of $\X^* \setminus F$ which contains $y_0$ and $y_1$. Let $G=\{y_0, y_1\}$.
Note that $F$ and $G$ satisfy the hypothesis of Lemma \ref{lemma-connected-separation2}.
So by this lemma there exist all the objects claimed by it.  
Work in $\beta X$ and consider
$$H=\bigcap_{k\in \N}(\overline{U_k'\cup U_k''\cup U_k'''}).$$
 First note that 
$S_{[0,a_k)}\cap(U_k'\cup U_k''\cup U_k''')=\emptyset$ by (a), (b) 
and (d) of Lemma \ref{lemma-connected-separation2}.
As $X=\bigcup_{k\in \N} S_{[0,a_k)}$ since $\lim_{k\rightarrow \infty}a_k=\infty$, we 
get that $H\subseteq X^*$. 
By Lemma \ref{lemma-connected-separation2} (g), we conclude that $G\subseteq H$.

So we are left with proving  that $H$ is a connected subset of $X^*$.
For this, by Lemma \ref{connected-union} (3),
it is enough to show that each $\overline{U_k'\cup U_k''\cup U_k'''}$ is connected and for this,
by Lemma \ref{connected-union} (2) it is enough to show
that $U_k'\cup U_k''\cup U_k'''$ are connected for every $k\in \N$. The rest
of the proof is devoted to this task.

Fix $k\in \N$. First note that the subspace $U_k'\cup U_k'''$ is connected. For this
use its definition in Lemma \ref{lemma-connected-separation2}, note that for $n\geq k$ each
$S_{(a_n, b_n)}$  is connected by Lemma \ref{remark-level-sets} (1)
and each $U_{n, m}'''=B_{\XX^*}\times U(l(\alpha_{n, m}, [b_n, a_{n+1}]), \epsilon_n)$
 is connected by Lemma \ref{remark-level-sets} (3).
Moreover, 
$$S_{(a_n, b_n)}\cap U_{n, m}'''\not=\emptyset\not=S_{(a_{n+1}, b_{n+1})}\cap U_{n, m}'''$$ for each $n\geq k$
by Lemma \ref{remark-level-sets} (2) since 
$$(a_n,b_n)\cap (b_n-\epsilon_n, a_{n+1}+\epsilon_n)\not=\emptyset\not=
(a_{n+1},b_{n+1})\cap (b_n-\epsilon_n, a_{n+1}+\epsilon_n)$$
and $B_{\XX^*}\times l(\alpha_{n, m}, (b_n-\epsilon_n, a_{n+1}+\epsilon_n)\subseteq U_{n,m}'''$.
So using Lemma \ref{connected-union} (1) it is easy to prove by induction that
 finite unions 
 $$S_{(a_k, b_k)}\cup \bigcup_{m\leq m_n}U_{k, m}'''\cup\dots \cup S_{(a_n, b_n)}\cup \bigcup_{m\leq m_n}U_{n, m}'''$$ 
 are connected,
and so is the entire union $U_k'\cup U_k'''$ by Lemma \ref{connected-union} (4).
 
 By the form of the sets
 $U_{n,m}''$ in Lemma \ref{lemma-connected-separation2} (b) they are connected for all $m\leq m_n$ and $n\in \N$.
  By Lemma \ref{lemma-connected-separation2} (d) and (f),
 $U_{n, m}''\cap U_{n, m}'''\not=\emptyset$ for each $n\in\N$ and $m\leq m_n$, and
 so $U_{n, m}''\cap (U_k'\cup U_k''') \neq \emptyset$ for each $n\geq k$ and $m\leq m_n$. 
 Hence, we can apply Lemma \ref{connected-union} (1)
to conclude that $U_k'\cup U_k''\cup U_k'''$  is connected as required.

\end{proof}

\begin{lemma} \label{lemma-abundant-connected}
Suppose that $\A\subseteq C(\X^*)$ is abundant. Then $\nabla\A\setminus F'$
 is connected for every finite $F' \subseteq \nabla \A$.
\end{lemma}
\begin{proof} Let $F=(\Pi_\A)^{-1}[F']$. As $\nabla\A\setminus F'$ is a continuous image
(under $\Pi_\A$)  of $\X^*\setminus F$, it is enough to show that $\X^*\setminus F$ is connected.
We would like to use Lemma \ref{big-connected}, and so we need to check if its hypothesis is satisfied. 
Let $(f_s)_{s\in \{0,1\}^{<\omega}}\subseteq C_1(X)$ be a Cantor system. 

Let $F'=\{x_1,\dots, x_j\}$ and $F_i=(\Pi_\A)^{-1}[\{x_i\}]$ for
$1\leq i\leq j$. As $\Pi_\A[F_i]=\{x\}$, clearly every $f\in \A$ is constant on
each $F_i$. Hence, by Lemma \ref{abundant-zero} there is
$f_\sigma$ such that $(\beta f_\sigma)|F=0$. So, Lemma \ref{big-connected} 
implies that $\X^*\setminus F$ is connected, and so is $\nabla\A\setminus F'$.
\end{proof}

We can now finally prove the main proposition of this section.

\begin{proposition}\label{main-connected} Suppose that $\XX$ is a Banach space of density not bigger than $\cc$
and  $X = B_{\XX^*}\times\R^2$, where the dual ball $B_{\XX^*}$ is considered with the weak$^*$
topology. There are pairwise disjoint families
$\{e^{\xi\alpha}: \alpha<\cc\}\subseteq C(X^*)$ for every $\xi<\cc$ such that
whenever a closed subalgebra $\A\subseteq C(X^*)$  satisfies 
$\A\cap\{e^{\xi\alpha}: \alpha<\cc\}\not=\emptyset$ for every $\xi < \cc$, then
$\nabla\A\setminus F$ is connected for every finite set $F\subseteq\nabla\A$.
\end{proposition}

\begin{proof}
    We will construct  $\{e^{\xi\alpha}: \alpha<\cc\}$, $\xi < \cc$, in such 
    a way that every closed subalgebra $\A  \subseteq C(X^*)$, which 
    intersects $\{e^{\xi \alpha}: \alpha < \cc\}$ for every $\xi < \cc$, is abundant. 
    The proposition will then follow from Lemma \ref{lemma-abundant-connected}.

      Note that there are $\cc$ many Cantor systems of functions in $C_1(X)$. 
      Indeed, $X$ is $\sigma$-compact and its weight  is at most $\cc$ and so the 
      cardinality of  $C(X)$ is at most $\cc$ by Lemma \ref{weight-remainder}.     
      As every Cantor system of functions in $C_1(\X)$ is a countable set of functions from $C_1(X)$, 
      we get that there are $\cc$ many such systems. 
      Let us enumerate them as $\{\mathcal{S}_\xi: \xi < \cc\}$. 
      For every $\xi < \cc$ we set $\{e^{\xi \alpha}: \alpha <\cc \}$
       to be any enumeration of $\{(\beta f_\sigma)|\X^*: \sigma \in \{0,1\}^\omega\}$, 
       where $\mathcal{S}_\xi$ is the Cantor system $(f_s)_{s \in \{0,1\}^{< \omega}}$. 
       Note that $\{(\beta f_\sigma) |X^*: \sigma \in \{0,1\}^\omega\}$ 
       is pairwise disjoint by Lemma \ref{cantor-disjoint}. 
       It follows that any $\A \subseteq C(X^*)$
        that intersects $\{e^{\xi \alpha}: \alpha < \cc\}$ for every $\xi < \cc$ satisfies
         the condition of Lemma \ref{abundant-equi}, and so is abundant.
\end{proof}

\section{The construction}

Let us recall again that every algebra is considered to be unital and not closed unless specifically stated.

\begin{proposition}\label{proposition-main}
Suppose that $M$ is a compact Hausdorff almost $P$-space of weight $\cc$ with the following property:
\begin{enumerate}
    \item[] For every pairwise disjoint sequence $(f_n)_{n\in \N}\subseteq C_1(M)$ and pairwise 
    disjoint cozero sets $(V_m)_{m\in \N}$ of $M$ satisfying $f_n|V_m=0$ for all $n,m\in \N$
     there is a submorphism $\phi: \wp(\N)\rightarrow C_1(M)$ satisfying
      $\phi(\N)|V_m=0$ and $\phi(\{n\})=f_n$ for every $n, m\in\N$.
\end{enumerate}

Let $\B$ be a closed subalgebra of $C(M)$ such that $\nabla\B$ does not
admit a subspace homeomorphic to $\beta\N$, let 
$\{e^{\xi\alpha}: \alpha<\cc\}\subseteq C_1(M)$ be pairwise disjoint for every $\xi<\cc$.

Then there is  a closed subalgebra $\A$ of $C(M)$ such that:
\begin{enumerate}[(i)]
\item $\nabla\A$ is an almost $P$-space.
\item The Banach space $C(\nabla\A)$ has few$^*$-operators\footnote{Note that
$C(\nabla\A)$ is isometrically, multiplicatively isomorphic to $\A$ by Proposition \ref{proposition-Pi} (2).}.
\item  $\{e^{\xi\alpha}: \alpha<\cc\}\cap\A\not=\emptyset$ for every
$\xi<\cc$.
\item  $\B  \subseteq \A$. 
\end{enumerate}
\end{proposition}

\begin{proof} 
 Since $M$ has weight $\cc$, we can enlarge the set 
$\{\{e^{\xi\alpha}: \alpha<\cc\}: \xi<\cc\}$ by some pairwise disjoint collections
 of size $\cc$ with supports below every open set from a fixed 
 open basis of $M$ of cardinality $\cc$. 
 Such collections can be formed in $C_1(M)$ by Lemma \ref{disjoint-P}. 
 The fact that $\nabla\A$ is an almost $P$-space will then follow from Lemma \ref{dense-P}.
 Hence (i) follows from (iii) with a modified $\{\{e^{\xi\alpha}: \alpha<\cc\}: \xi<\cc\}$.

Before going to the construction of $\A$ let us fix some enumerations.
Let us write $$\mathfrak{c} = 
H_1\cup H_2\cup\bigcup_{\alpha<\mathfrak{c}}F_\alpha \cup \bigcup_{\alpha<\mathfrak{c}}G_{\alpha}$$
as a disjoint union where $|H_1|=|H_2|=|F_\alpha| = |G_{\alpha}| = \mathfrak{c}$, 
$0\in H_1$ and $F_\alpha \cap \alpha = G_{\alpha}\cap \alpha =\emptyset$ for each $\alpha<\mathfrak{c}$. 
We write $G_{\alpha} = \{\xi_{\alpha}^{\beta\gamma} : \beta, \gamma<\mathfrak c\}$.
We fix enumerations
\begin{enumerate}
    \item $\{A^\gamma: \gamma<\cc\}$ of all infinite subsets of $\N$,
    \item $\{h^{\xi}: \xi\in H_1\}$ of all elements  of $\B$ with $h^0$ a nonzero constant function.
\end{enumerate}
The enumeration as in (2) is possible by Lemma \ref{weight-remainder}.
Moreover we re-enumerate  
\begin{enumerate}
    \item[(3)] $\{\{e^{\xi\alpha}: \alpha<\cc\}: \xi<\cc\}$ as $\{\{e^{\xi\alpha}: \alpha<\cc\}: \xi\in H_2\}$.
\end{enumerate}

 We construct $\A$ as the
 union of a continuous increasing chain $(\mathcal A_\xi)_{\xi<\cc}$
of (not necessarily closed) subalgebras $\mathcal A_\xi$ of $C(M)$ of the densities smaller than $\cc$. 
Once the algebra $\mathcal{A}_\alpha$ for $\alpha<\cc$ is defined, we make the further enumerations:    
\begin{enumerate}
    \item[(4)] $\{h^{\xi\alpha} : \xi\in F_\alpha\}$ are all elements of the norm closure ${\overline{\A_\alpha}}$
    of $\A_\alpha$.
    \item[(5)] $\{(\{V_m^\beta(\alpha): {m\in\N}\}, \{f_n^\beta(\alpha) : {n\in \N}\}): \beta<\cc\}$ are all possible pairs such that
    \begin{enumerate}

        \item $\{V_m^\beta(\alpha): {m\in \N}\}$ is a family of pairwise disjoint  nonempty sets in $\coz(\A_\alpha)$,
        \item $\{f_n^\beta(\alpha) : {n\in \N}\}$ is a pairwise disjoint family in $(\A_\alpha)_{1}$ such that
         $f_n^\beta(\alpha)|V_m^\beta(\alpha)=0$
         for all $n \in \N$ and $m \in \N$.
        \item We fix for each $\beta < \cc$ submorphisms $\phi^\beta_\alpha: \wp(\N)\rightarrow C_1(M)$ 
        such that $\phi^\beta_\alpha(\{n\}) = f_n^\beta(\alpha)$ and $\phi^\beta_\alpha(\N)|V_m^\beta(\alpha)=0$
         for all $m\in \N$. Such submorphisms exist by the hypothesis of the lemma.
    \end{enumerate}
    Such enumerations are possible by a simple counting argument as the number 
    $|\coz(\A_\alpha)|\leq |\A_\alpha|\leq\cc$ and  $|\A_\alpha|^\omega\leq \cc^\omega=\cc$.
\end{enumerate}

Now we describe the transfinite recursive construction of $(\mathcal{A}_\xi)_{\xi<\cc}$. 
To be able to complete each stage $\xi<\cc$ of the construction we will need to know
that certain conditions about certain objects hold for the previous stages, so such 
objects are recursively introduced and the conditions are inductively proved. The objects are:
\begin{enumerate}
    \item[(6)] Countable sets $D_\xi, E_\xi\subseteq M$ such that $D_\xi\cap E_\xi=\emptyset$ for 
    $\xi\in\bigcup_{\alpha<\cc}G_\alpha$.
    \item[(7)] The closed subalgebra $\C_\xi$ of $C(M)$ generated by $\A_\xi\cup\B$  for $\xi < \cc$.
\end{enumerate}
The inductive conditions are
\begin{enumerate}
    \item[$(*_\xi)$] For every  $\xi'\in \xi\cap\bigcup_{\alpha<\cc}G_\alpha$ the sets 
    $D_{\xi'}$ and $E_{\xi'}$ are not separated by $\C_\xi$.

\end{enumerate}

We put $\A_0=\{0\}$.  So $(*_0)$ holds since there is no  $\xi'<\xi$.
At limit steps we shall define $\mathcal{A}_\xi = \bigcup_{\eta<\xi}\mathcal{A}_\eta$. 
Then $(*_\xi)$ holds by the inductive assumption and Lemma \ref{nonseparation-limit}.
At a successor stage, $\mathcal{A}_{\xi+1}$ will be the algebra generated by $\mathcal{A}_\xi$ 
and a certain element $g_\xi\in C(M)$ that we will add, $\mathcal{A}_{\xi+1} = \mathcal{A}_\xi\langle g_\xi\rangle$. 
Below we describe $g_\xi$ depending on what is $\xi$. Now that in the first three cases below
$\xi\not\in \bigcup_{\alpha<\cc}G_\alpha$ and so $D_\xi$ and $E_\xi$ are not defined by (6).
\begin{itemize}
    \item When $\xi \in H_1$, then $g_\xi$ will be $h^{\xi}$. In particular, $\A_1$ 
    is generated by the nonzero constant function $h^0$. Since $h^\xi \in \B$, it follows that $\C_{\xi+1}=\C_\xi$,
    so $(*_{\xi+1})$ 
      follows from $(*_\xi)$.
    \item When $\xi \in H_2$, then  $g_\xi$ will be one of the functions from the pairwise 
    disjoint family $\{e^{\xi\alpha}: \alpha<\cc\}$. 
    The union of all the sets $E_{\xi'}, D_{\xi'}$ for $\xi'\in \xi\cap\bigcup_{\alpha<\cc}G_\alpha$
     has cardinality smaller than $\cc$ (it is a union of less than $\cc$ countable sets).
      Hence, we can use the fact that $\{e^{\xi \alpha} : \alpha < \cc\}$ are pairwise 
      disjoint to find $\alpha<\cc$ such that $e^{\xi\alpha}|(E_{\xi'}\cup D_{\xi'})=0$ 
      for every $\xi'\in \xi\cap\bigcup_{\beta<\cc}G_\beta$. Then $(*_{\xi+1})$ 
      follows from $(*_\xi)$ and Lemma \ref{nonseparation-zero}.
    \item When $\xi \in F_\alpha$ for some $\alpha < \cc$, $g_\xi$ will 
    be $h^{\xi\alpha}$ (as $F_\alpha \cap \alpha = \emptyset$, 
    we get that $\alpha \leq \xi$, and $h^{\xi \alpha}$ has already been defined). 
    Then $\C_{\xi+1} = \C_\xi$ and $(*_{\xi+1})$ follows from $(*_\xi)$.
    \item When $\xi = \xi^{\beta\gamma}_{\alpha}\in G_\alpha$ 
    for some $\alpha<\xi$ then we consider $A^\gamma\subseteq\N$ 
    and pairwise disjoint families $\{f_n^\beta(\alpha): n\in\N\}\subseteq \A_\alpha$
     and $\{V_m^\beta(\alpha) : m\in\N\}\subset \coz(\A_\alpha)$ 
     which satisfy $f_n^\beta(\alpha)|V_m^\beta(\alpha)=0$ for every 
     $n,m\in \N$. We also consider the submorphism $\phi^\beta_\alpha$ as in (5c).
  
    We can apply Lemma~\ref{haydon} for the algebra $\CC_\xi$, for the family $\{f_n^\beta(\alpha), n\in A^\gamma\}$ 
    for the submorphism $\phi^\beta_\alpha|\wp(A^\gamma)$ and for the sets 
    $E_{\xi'}$, $D_{\xi'}$ for $\xi'\in \xi\cap\bigcup_{\alpha<\cc}G_\alpha$. 
    Indeed, $\CC_\xi\supseteq\A_\alpha$ (again, as $G_\alpha \cap \alpha = \emptyset$ 
    we get that $\alpha \leq \xi$) does not separate any of the pairs $E_{\xi'}$, 
    $D_{\xi'}$ by ($*_\xi$). Lemma \ref{haydon} provides us with an infinite 
    $A\subseteq A^\gamma$ such that for every infinite $B\subseteq A$ \\
    
    \begin{itemize}
        \item[(+)] The condition ($*_\xi$) holds in the algebra generated by $\C_\xi$ and $\phi_\alpha^\beta(B)$. \\
    \end{itemize}
 
    We will choose one of such $B$'s having in mind the preparation of new sets $E_\xi$ and $D_\xi$ to obtain $(*_{\xi+1})$. 
    Namely, we apply Lemma \ref{new-promise} to find $B^\xi \subseteq A$ and countable 
    $E_\xi, D_\xi\subseteq M$ such that $D_\xi$ and $E_\xi$ are not separated by 
    $\CC_\xi$ and $D_\xi \subseteq \bigcup\{V_m^\beta(\alpha) : m\in B^\xi\}$  and
    $E_\xi\subseteq \bigcup\{V_m^\beta(\alpha) : m\in A\setminus B^\xi\}$
     and $|D\cap V_m|=|E\cap V_{m'}|=1$ for every $m\in B^\xi$ and every $m'\in A\setminus B^\xi$. 
     The hypothesis of Lemma \ref{new-promise} is satisfied by Lemma \ref{betaN-product} 
     as the density of $\A_\xi$ is less then $\cc$, and so the weight of 
     $\nabla\A_\xi$ is less than $\cc$ and since by the hypothesis of the proposition $\nabla\B$ 
     does not contain a copy of $\beta\N$, so neither does $\nabla \CC_\xi$.

    Finally, define $g_\xi$ in this case of $\xi = \xi^{\beta\gamma}_{\alpha}\in G_\alpha$ 
    for some $\alpha<\xi$ by $g_\xi = \phi(B^\xi)$ and $\mathcal{A}_{\xi+1} = \mathcal{A}_\xi\langle g_\xi\rangle$. 
    To prove ($*_{\xi+1}$) use (+), the fact that $\phi_\alpha^\beta(\N)|V_m^\beta(\alpha)=0$ 
    by the choice in (5c), which implies that $\phi_\alpha^\beta(B^\xi)|V_m^\beta(\alpha)=0$
     by Lemma \ref{morphism-monotone} which implies that 
     $\phi_\alpha^\beta(B^\xi)|(D_\xi\cup E_\xi)=0$. Finally, we apply Lemma \ref{nonseparation-zero}
     to conclude ($*_{\xi+1}$).
\end{itemize}
     
This completes the description of the construction of $\A=\bigcup_{\xi<\cc} \A_\xi$ 
and the proof of ($*_\xi$) for every $\xi<\cc$, which by Lemma \ref{nonseparation-limit} 
implies that $\A$ does not separate $D_\xi$ and $E_\xi$ for every $\xi\in \bigcup_{\alpha<\cc} G_\alpha$.

Now we need to prove the properties of $\A$ claimed in the proposition.
The case of $\xi\in H_1$ takes care of (iv).
The case of $\xi\in H_2$ takes care of (iii)
and as noted at the beginning of the proof this also gives (i).

Further, $\A$ is a subalgebra of $C(M)$ because it is an increasing union of subalgebras 
of $C(M)$ and $\A$ is closed because at stages $\xi$ of the construction for 
$\xi\in F_\alpha$ for some $\alpha<\xi$ we obtain 
$\overline{\A_\alpha}\subseteq \A$ for every $\alpha<\cc$. 
Since the cofinality of $\cc$ is uncountable and the sequence 
$(\A_\alpha)_{\alpha<\cc}$ is increasing this implies that the limit of any sequence from $\A$ is in $\A$.

Finally, let us prove that $C(\nabla\A)$ has few$^*$-operators as in (ii).
 For this we check the hypothesis of Proposition \ref{few-star} 
 for $L = \nabla \A$. By Proposition \ref{proposition-Pi} (1) we may
 identify elements   $f\in C(\nabla\A)$ with elements $f\circ\Pi_\A\in \A$ which are all of this form.
 
 We already know that $\nabla\A$ is an almost $P$-space. 
 So, suppose that $\{V_m: m\in \N\}$ is a family of pairwise disjoint 
 nonempty sets in $\coz(\A)$ and that $\{f_n : {n\in \N}\}\subseteq (\A)_1$
   is pairwise disjoint and such that $f_n|V_m=0$ for every $m, n\in \N$. 
   Again, since the cofinality of $\cc$ is uncountable, there is $\alpha<\cc$ 
   such that $\{V_m: m\in\N\}\subseteq\coz(\A_\alpha)$ and 
   $\{f_n : {n\in \N}\}\subseteq \A_\alpha$. 
   It follows from (5) that there is $\beta<\cc$ such that $V_m=V_m^\beta(\alpha)$ and 
   $f_n=f_n^\beta(\alpha)$ for every $n, m\in \N$. 
   We claim that $\mathcal P=(\phi^{\beta}_{\alpha})^{-1} [\A]$ 
   and $\phi^\beta_\alpha|\mathcal P$ satisfy
    (a) - (d) of Proposition \ref{few-star}, and so $C(\nabla\A)$ has few$^*$-operators.

Note that $\mathcal{P} = (\phi^{\beta}_{\alpha})^{-1} [\A]$ is a sublattice of $\wp (\N)$
containing the least element $\emptyset$ by Lemma \ref{sublattice}. Also $\mathcal{P}$ 
contains finite sets as $\phi^\beta_\alpha(\{n\}) = f^\beta_n(\alpha) \in \A$ for every $n \in \N$
and finite sets are suprema of its singletons while $\A$ must contain
finite sums of its elements as an algebra, moreover  the sums of pairwise disjoint elements in $(\A)_1$ are 
suprema of corresponding finite families.

Further, if $A \in \mathcal{P}$ and $k \in \N$, then $A\cap k\in \mathcal P$ (as $A\cap k$ is finite)
and $A\setminus k$ are
disjoint and so are $\phi^\beta_\alpha(A\cap k)$ and $\phi^\beta_\alpha(A\setminus k)$
by Lemma \ref{disjoint-preserv}. As the suprema of disjoint elements of $(\A)_1$ are their sums we have

$$\phi^\beta_\alpha(A) = \phi^\beta_\alpha(A\cap k) + \phi^\beta_\alpha(A \setminus k),$$
and thus $\phi^\beta_\alpha(A \setminus k) = 
\phi^\beta_\alpha(A) - \phi^\beta_\alpha(A\cap k)  \in \A$
 and we have that $A \setminus k \in \mathcal{P}$. We have thus proved (a) of Proposition \ref{few-star}.
 
Condition (b) of Proposition \ref{few-star} is clear by the choice of  $\phi^\beta_\alpha$
and condition (c) of Proposition \ref{few-star} follows from the fact that 
$\phi^\beta_\alpha(\N)|V_m^\beta(\alpha) = 0$ for every $m \in \N$ and from Lemma \ref{morphism-monotone}. 

To prove (d) of Proposition \ref{few-star} pick any infinite $A\subseteq\N$. 
Then there is $\gamma<\cc$ such that $A=A^\gamma$. 
Let us observe what happened in the construction at stage 
$\xi=\xi_\alpha^{\beta\gamma}$. 
We found an infinite $B^\xi\subseteq A^\gamma=A$ 
and added $\phi_\alpha^\beta(B^\xi)$ to the algebra $\A_\xi$ forming $\A_{\xi+1}\subseteq\A$, 
moreover $\A$ does not separate $E_\xi$ and $D_\xi$ which
 implies that it does not separate their supersets $\bigcup\{V_m^\beta(\alpha) : m\in B^\xi\}$ and
  $\bigcup\{V_m^\beta(\alpha) : m\in A^\gamma\setminus B^\xi\}$ 
  which proves that (d) of Proposition \ref{few-star} is satisfied.
  So (ii) is satisfied which completes the proof of the proposition.
\end{proof}

\begin{proposition}\label{cor-topo} Suppose that $\XX$ is Banach space of density at most $\cc$
such that the dual ball $B_{\XX^*}$ considered with the weak$^*$ topology does not
admit a subspace homeomorphic to $\beta\N$.
Then there is a compact Hausdorff $K$ of weight at most continuum and a continuous surjection
 $\phi: K\rightarrow B_{\XX^*}$
such that $C(K)$ is an indecomposable Banach space with few$^*$ operators. In particular $C(B_{\XX^*})$, is
isometrically multiplicatively linearly isomorphic  to a closed subalgebra of $C(K)$.
\end{proposition}
\begin{proof} We would like to use Proposition \ref{proposition-main} for
$M$ equal to the \v Cech-Stone remainder of $B_{\XX^*} \times \R^2$, i.e., 
$M=X^*$ with $X=B_{\XX^*} \times \R^2$.
    $M$ is an almost $P$-space 
    by Lemma \ref{remainder-P-space} and is of weight at most $\cc$ by Lemma \ref{weight-remainder}
    and  satisfies the remaining part of the hypothesis 
     of Proposition \ref{proposition-main}  by Lemma \ref{morphism-disjoint}. 
     
    Now for the use of Proposition \ref{proposition-main} we need a closed  subalgebra $\B$ of $C(M)$ such that
    $C(B_{\XX^*})$ isometrically embeds into $\B$ and $\nabla\B$ does not admit a subspace homeomorphic to $\beta\N$.
    For this note that $B_{\XX^*}$  is a continuous 
    image of $M=X^*$. Indeed, the sought continuous surjection can be $\pi^* = (\beta \pi)| X^*:M\rightarrow B_{\XX^*}$, 
    where $\pi: X \rightarrow B_{\XX^*}$ is the projection on the first coordinate. 
    Indeed, $\pi^*$ is clearly continuous an into $B_{\XX^*}$, 
    and it is also onto as for any $x \in B_{\XX^*}$ the 
    preimage $\pi^{-1}[\{x\}]$ is a subset of $X$ which is not compact, and thus any
     $x^* \in \overline{\pi^{-1}[\{x\}]} \setminus X \neq \emptyset$, 
     where the closure is taken in $\beta X$, satisfies $\pi^*(x^*) = x$.
     
     So define
     $$\B=\{f\circ\pi^*: f\in C(B_{\XX^*})\}\subseteq C(M).$$
    Then  $\nabla \B$ is
    homeomorphic to $B_{\XX^*}$  by Lemma \ref{homeo}
     and so does not admit a subspace homeomorphic to $\beta \N$. 
     We can thus apply Proposition \ref{proposition-main} for  this $\B$
      with $\{e^{\xi \alpha}: \alpha < \cc\}$, $\xi < \cc$, given by Proposition \ref{main-connected}
      obtaining $\A\subseteq C(M)$ satisfying Proposition \ref{proposition-main} and  Proposition \ref{main-connected}.  
      
      We claim that $K=\nabla\A$ satisfies the conclusion of the proposition.
      By Proposition \ref{proposition-main} (ii) the Banach space $C(\nabla\A)$ has few$^*$ operators
      and by Proposition \ref{proposition-main} (iii) and Proposition \ref{main-connected} the
      compact space $\nabla\A\setminus F$ is connected for each finite $F\subseteq\nabla\A$.
      This implies that $C(\nabla\mathcal{A})=C(K)$ is an indecomposable Banach space by Theorem 2.5. of \cite{few}.  
      Moreover by Lemma \ref{weight-remainder}  the weight of $X^*=M$  and so of $K$
      is at most continuum. 
      
      As $\B\subseteq\A\subseteq C(M)$ by Proposition \ref{proposition-main} (iv),
      we conclude that $\nabla\B$ is a continuous image
      of $\nabla\A=K$, the map being the canonical projection 
      from $\R^{\A}$ onto $\R^\B$ restricted to $\nabla\A$. Since $\nabla \B$ is
    homeomorphic to $B_{\XX^*}$  by Lemma \ref{homeo} we obtain that
    $B_{\XX^*}$ is a continuous image of $K=\nabla\A$ as required and 
    hence there is a linear multiplicative isometric embedding of  $C(B_{\XX^*})$  into $C(K)$
    by Proposition \ref{stone} (1).
      
\end{proof}

\section{Indecomposable spaces with $\ell_\infty$  as quotients}

\begin{definition}\label{def-ssp}
    Let $K$ be a compact Hausdorff space. We say that $C(K)$ has the weak subsequential 
    separation property if for every pairwise disjoint sequence $(f_n)_{n \in \N}$ in $C_1(K)$
     there is $f \in C_1(K)$ such that both of the sets $\{n \in \N: f_n \leq f\}$ and $\{n \in \N: f_n \wedge f = 0\}$ are infinite.
\end{definition}

Weak subsequential separation property of $C(K)$ spaces is motivated by the analogical property 
for Boolean algebras introduced in \cite{alg-universalis} -- a Boolean algebra 
$\A$ has the weak subsequential separation property if for every pairwise disjoint sequence 
$(A_n)_{n \in \N}$ of elements of $\A$ there is $A \in \A$
 such that both of the sets $\{n \in \N: A_n \leq A\}$ and $\{n \in \N: A_n \wedge A = 0\}$ are infinite. 
 This property is a generalization of subsequential completeness property,
  and was shown (see Corollary 1.5. of \cite{alg-universalis}) 
  to imply the existence of a copy of $\beta \N$ in the Stone space $K_\A$ of $\A$.
   We will show that the ideas of \cite{alg-universalis} can be adapted to the case when $K$ is not totally disconnected as well.

\begin{proposition}\label{weak-subseq-prop}
    Let $K$ be a compact Hausdorff space such that $C(K)$ has the weak subsequential separation property. 
    Then $K$ contains a copy of $\beta \N$. In particular, $C(K)$ admits $\ell_\infty$ as its quotient.
\end{proposition}

\begin{proof}
    First we note that $K$ cannot contain a nontrivial convergent sequence -- if $(x_n)_{n \in \N}$ 
    was an injective sequence in $K$ converging to some $x \in K$, 
    then we could find a pairwise disjoint sequence $(f_n)_{n \in \N}$ in $C_1(K)$ with $f_n(x_n)=1$ 
    for $n \in \N$. Applying weak subsequential separation property to this sequence would 
    give us $f \in C_1(K)$ such that $f(x_n) = 1$ for infinitely many $n \in \N$ and also $f(x_n) = 0$
     for infinitely many $n \in \N$, which contradicts continuity of $f$. 
     Hence, $K$ cannot be scattered and so contains a nonempty 
     perfect set (e.g. by 1.7.10. of \cite{engelking}).  So we may assume that
     $K$ has no isolated points, as it is enough to find a copy of $\beta\N$ in a subspace of $K$.
     
     We can thus use the normality of $K$ and the hypothesis that $K$ 
     has no isolated points to see that that there is a family $(U_s, V_s)_{s \in \{0,1\}^{< \omega}}$
      of pairs of nonempty open sets in $K$ such that:
    \begin{enumerate}[(i)]
        \item $U_{s \frown (i)} \subseteq \overline{U_{s \frown (i)}} \subseteq V_s \subseteq \overline{V_s} \subseteq U_s$
         for every $s \in \{0,1\}^{< \omega}$ and $i \in \{0,1\}$.
        \item $U_{s \frown (0)} \cap U_{s \frown (1)} = \emptyset$ for every $s \in \{0,1\}^{< \omega}$.
    \end{enumerate}
    We fix for $s \in \{0,1\}^{< \omega}$ a function $g_s \in C_1(K)$ such that
     $g_s|V_s = 1$ and $g_s|(K \setminus U_s) = 0$. 
     It is readily proved that for incomparable $s,t \in \{0,1\}^{<\omega}$
      the functions $g_s$ and $g_t$ are disjoint, and if $s_1,\dots,s_k \in \{0,1\}^{<\omega}$ 
      are pairwise comparable, then $(g_{s_1} \wedge \cdots \wedge g_{s_k}) | \bigcap_{j=1}^k V_{s_j} = 1$.

    It follows from the proof of Theorem 1.4. of \cite{alg-universalis} (see (*) and (**) on page 306)
     that there is a system $(A_s)_{s \in \{0,1\}^{<\omega}}$ of clopen subsets of the Cantor set such that
    \begin{enumerate}[(a)]
        \item $A_s \cap A_t = \emptyset$ if $s \subsetneq t$,
        \item $A_{s_1} \cap \cdots \cap A_{s_k} \neq \emptyset$ if $s_1,\dots,s_k$ are pairwise incomparable. 
    \end{enumerate}
    We denote by $B_s$, for $s \in \{0,1\}^{<\omega}$, the basic clopen subset of the Cantor set, 
    $B_s = \{\sigma \in \{0,1\}^\N: s \text{ is an initial segment of } \sigma\}$. 
    These sets have the property that $B_s \cap B_t = \emptyset$ if $s,t$ are incomparable and 
    $B_s \cap B_t = B_s$ if $t \subsetneq s$.
    Moreover, any clopen subset of the Cantor set is a finite union of these basic sets, 
    and in particular for every $s \in \{0,1\}^{<\omega}$ there is $k_s\in\N$ such that
    \begin{align*}
        A_s = B_{t(s,1)} \cup \cdots \cup B_{t(s,k_s)} \text{ for some } t(s,1),\dots,t(s,k_s) \in \{0,1\}^{<\omega}.
    \end{align*}
    For $s \in \{0,1\}^{<\omega}$ set $f_s = \bigvee_{i=1,\dots,k_s} g_{t(s,i)} \in C_1(K)$. 
    We check that the system $(f_s)_{s \in \{0,1\}^{<\omega}}$ has properties analogous to (a) and (b), namely:
    \begin{enumerate}[(A)]
        \item $f_s$ and  $f_t$ are disjoint if $s \subsetneq t$,
        \item If $s_1,\dots,s_l$ are pairwise incomparable, 
        then there is $x \in K$ such that $f_{s_1}(x) = \dots = f_{s_l}(x) = 1$.
    \end{enumerate}
    To prove (A) note that if $s \subsetneq t$ then $A_s \cap A_t = \emptyset$ by (a). 
    Hence, for every $i=1,\dots,k_s$ and $j=1,\dots,k_t$ we have 
    $B_{t(s,i)} \cap B_{t(t,j)} = \emptyset$ and $t(s,i)$ and $t(t,j)$ are incomparable. 
    It follows that $f_s$ and $f_t$ are disjoint and we have (A).
    Now we show (B). Let $s_1,\dots,s_l \in \{0,1\}^{<\omega}$ be incomparable. 
    Then $A_{s_1} \cap \cdots \cap A_{s_l} \neq \emptyset$, and hence there are 
    $i_1,\dots,i_l$ such that $B_{t(s_1,i_1)} \cap \cdots \cap B_{t(s_l,i_l)} \neq \emptyset$. 
    But then $t(s_1,i_1),\dots,t(s_l,i_l)$ are pairwise comparable and there is $x \in K$ such 
    that $g_{t(s_1,i_1)}(x) = \dots = g_{t(s_l,i_l)}(x) = 1$ (we can pick any $x \in \bigcap_j V_{t(s_j,i_j)} \neq \emptyset$). 
    Hence, (B) is true as $1 \geq f_{s_j}(x) \geq g_{t(s_j,i_j)}(x)$ for every $j=1,\dots,l$. We are done with proving (A) and (B).

    For $\sigma \in \{0,1\}^\N$ the sequence $(f_{\sigma|n})_{n \in \N}$ is pairwise disjoint by (A), 
    and hence we can use the weak subsequential separation property to find disjoint infinite sets 
    $N_\sigma, M_\sigma \subseteq \N$ and $f_\sigma \in C_1(K)$ such that $f_{\sigma|n} \leq f_\sigma$ for $n \in N_\sigma$
     and $f_{\sigma|n} \wedge f_\sigma= 0$ for $n \in M_\sigma$. We will show that if we 
     set $C_\sigma = f_\sigma^{-1}[\{0\}]$ and $D_\sigma = f_\sigma^{-1}[\{1\}]$, 
     then $(C_\sigma, D_\sigma)_{\sigma \in \{0,1\}^\N}$ 
     is independent (that is, for any finite disjoint sets $F,G \subseteq \{0,1\}^{\N}$ 
     we have $\bigcap_{\sigma \in F} C_\sigma \cap \bigcap_{\sigma \in G} D_\sigma \neq \emptyset$). 
     Fix finite disjoint sets $F,G \subseteq \{0,1\}^{\N}$ and find $m \in \N$ 
     large enough so that the family $\{\sigma|m: x \in F \cup G\}$ is pairwise distinct. 
     For every $\sigma \in F$ find $m_\sigma\geq m$ such that $m_\sigma \in N_\sigma$
      and for every $\sigma \in G$ find $m_\sigma \geq m$ such that $m_\sigma \in M_\sigma$. 
      Then $\{\sigma|m_\sigma: \sigma \in F \cup G\}$ is pairwise incomparable and we can use
       (B) to find $x \in K$ such that $f_{\sigma|m_\sigma}(x) = 1$ for every $\sigma \in F \cup G$. 
       If $\sigma \in F$, then $m_\sigma \in N_\sigma$ and $f_\sigma(x) \geq f_{\sigma|m_\sigma}(x) = 1$, 
       so $x \in C_\sigma$. If $\sigma \in G$, then $m_\sigma \in M_\sigma$ and $f_\sigma(x) = 0$ 
       (as $f_{\sigma|m_x}(x) = 1$ and $f_\sigma$ is disjoint with $f_{\sigma|m_\sigma}$), so $x\in D_\sigma$. 
       Hence, $x\in \bigcap_{\sigma \in F} C_\sigma \cap \bigcap_{\sigma \in G} D_\sigma$ 
       and $(C_\sigma, D_\sigma)_{\sigma \in \{0,1\}^{\N}}$ is independent.
    
    Let $\phi: K \rightarrow [0,1]^{\{0,1\}^\N}$ be the continuous map defined by 
    $\phi(x) = (f_\sigma(x))_{\sigma \in \{0,1\}^\N}$ for $x \in K$. 
    Then $\phi[K]$ is a closed subset of $[0,1]^{\{0,1\}^\N}$
     which contains $\{0,1\}^{\{0,1\}^\N}$. 
     Indeed, it is closed as it is a continuous image of a compact set $K$. 
     To see that it contains $\{0,1\}^{\{0,1\}^\N}$ pick $u \in \{0,1\}^{\{0,1\}^\N}$
      and a basic neighborhood $V$ of $u$ of the form 
      $V = \{v \in [0,1]^{\{0,1\}^\N}: |v(\sigma) - u(\sigma)| < \epsilon \text{ for all } \sigma \in H\}$
       for some finite set $H \subseteq {\{0,1\}^\N}$ and $\epsilon > 0$. Set $F = \{\sigma \in H: u(\sigma) = 1\}$ 
       and $G = \{\sigma \in H: u(\sigma) = 0\} = H \setminus F$. 
       By the independence of $(C_\sigma, D_\sigma)_{\sigma \in \{0,1\}^{\N}}$ 
       there is some $x \in \bigcap_{x \in F} C_\sigma \cap \bigcap_{\sigma \in G} D_\sigma$. 
       Then $\phi(x) \in \phi[K] \cap V$. It follows that $u \in \overline{\phi[K]} = \phi[K]$.
        As the cardinality of $\{0,1\}^\N$ is $\cc$, 
        we get that that there is a closed subset $L$ of $K$ which continuously maps onto
         $\{0,1\}^\cc$, and this is known to be equivalent to $K$ 
         containing a copy of $\beta \N$ (Lemma \ref{projectivity}). 
\end{proof}

The following proposition provides a sufficient and purely topological condition
for a compact space $K$ to contain a copy of $\beta\N$.

\begin{proposition}\label{subseq-top}
    Let $K$ be a compact Hausdorff space with the following
    property:   whenever $U_n$ for $n\in \N$
    are pairwise disjoint cozero sets, then there are disjoint infinite
    $A, A'\subseteq\N$ such that
    $$\overline{\bigcup_{n\in A}U_n}\cap \overline{\bigcup_{n\in A'}U_n}=\emptyset.$$
    Then $C(K)$ has the weak subsequential 
    separation property. In particular, $K$ admits a subspace homeomorphic to
    $\beta\N$.
\end{proposition}
\begin{proof} Let $f_n\in C_1(K)$ be pairwise disjoint for $n\in \N$ and let
$U_n=f_n^{-1}[\R\setminus\{0\}]$. By the property of $K$ from the hypothesis of the proposition and by the normality of $K$
there are disjoint infinite $A, A'\subseteq\N$ and there is $f\in C_1(K)$ such that 
$$f|\big(\overline{\bigcup_{n\in A}U_n}\big)=1\ \hbox{and}\ f|\big(\overline{\bigcup_{n\in A'}U_n}\big)=0.$$
It follows that $f$ has the properties desired for the weak subsequential 
    separation property.
\end{proof}

\section{Acknowledgements}

The authors would like to thank S. Argyros for pointing out the results from \cite{argyros-r, argyros-etal}
and N. de Rancourt for pointing out the relevance of \cite{bourgain} which allows to conclude  Remark \ref{noe}.
We would also like to thank the referees for helpful comments.

\bibliographystyle{amsplain}

\end{document}